\documentclass[final,leqno,onefignum,onetabnum]{siamltex1213}
\usepackage[utf8]{inputenc}
\usepackage{graphicx,color}
\usepackage{amsmath}
\usepackage{amssymb}
\usepackage{psfrag}
\usepackage{graphics} 
\usepackage{epsfig} 
\usepackage{amsbsy}
\usepackage{textcomp}
\usepackage[section]{placeins} 
\newtheorem{example}{Example}[section]

\title{Necessary Condition on Lyapunov Functions Corresponding to the Globally Asymptotically Stable Equilibrium Point}
\author{Chirayu D. Athalye, Harish K. Pillai, and Debasattam Pal\thanks{Chirayu D. Athalye (\email{chirayu@ee.iitb.ac.in}) and Harish K. Pillai (\email{hp@ee.iitb.ac.in}) are with the Department of Electrical Engineering,
Indian Institute of Technology Bombay, Mumbai 400076, India. Debasattam Pal (\email{debasattam@iitg.ernet.in}) is with the Department of Electronics and Electrical Engineering, Indian Institute of Technology Guwahati, Guwahati 781309, India.}}

\begin{document}
\maketitle
\slugger{sicon}{xxxx}{xx}{x}{x--x}

\begin{abstract}
It is well known that, the existence 
 of a Lyapunov function is a sufficient condition for stability, asymptotic stability, or global
 asymptotic stability of an equilibrium point of an autonomous system $\dot{\mathbf{x}}=f(\mathbf{x})$.
 In variants of Lyapunov theorems, the condition for a Lyapunov candidate $V$ (continuously differentiable 
 and positive definite function) to be a Lyapunov function is that its time derivative along system
 trajectories, i.e. $\dot{V}(\cdot)=\langle \nabla V(\cdot),f(\cdot) \rangle$, must be negative semi-definite 
 or negative definite. Numerically checking positive definiteness of $V$ is very difficult; checking negative definiteness of  $\dot{V}(\cdot)$ is even more difficult, because it involves dynamics of the system. \\
 \indent We give a necessary 
 condition independent of the system dynamics, for every Lyapunov function corresponding to the globally asymptotically stable equilibrium point of $\dot{\mathbf{x}}=f(\mathbf{x})$. This necessary condition is numerically easier to check than checking positive definiteness of a function. Therefore, it can be used as a first level test to check whether a given continuously differentiable function is a Lyapunov function candidate or not. We also propose a method, which we call a generalized steepest descent method, to check this condition numerically. Generalized steepest descent method can be used for ruling out Lyapunov candidates corresponding to the globally asymptotically stable equilibrium point of $\dot{\mathbf{x}}=f(\mathbf{x})$. It can also be used as a heuristic to check the local positive definiteness of a function, which is a necessary condition for a Lyapunov function corresponding to a stable and/or asymptotically stable equilibrium point of an autonomous system.
\end{abstract}
\pagestyle{myheadings}
\thispagestyle{plain}
\markboth{Chirayu D. Athalye}{USING SIAM'S \LaTeX\ MACROS}
\section{Introduction}
\indent Stability analysis is a very crucial topic in systems theory. There are different kinds 
of stability problems (e.g: stability of equilibrium points, stability of periodic orbits, 
input-output stability etc.) that arise in the study of dynamical systems. Stability of the 
equilibrium points of an autonomous system is characterized using Lyapunov theory. Existence 
of a Lyapunov function is a sufficient condition for stability of an equilibrium point of an 
autonomous system. However for a large class of autonomous systems, finding a Lyapunov function 
may not be an easy task.\\ 
\indent In case of linear systems, the problem of finding a Lyapunov function reduces to solving 
a simple Linear Matrix Inequality (LMI); and hence there is a systematic way to find a Lyapunov 
function. Also in case of electrical and mechanical systems, there are natural Lyapunov function 
candidates in terms of physical energy functions \cite{khalil}. However in general for non-linear 
systems, there is no systematic method to find a Lyapunov function \cite{khalil}. In theory there 
is a variable gradient method; but for many autonomous systems, the variable gradient method is
extremely difficult to apply. \\
\indent In \cite{pablo} a technique is given for an algorithmic construction of a Lyapunov function 
for non-linear autonomous systems, with polynomial vector fields, using semidefinite programming 
and sum of squares decomposition. In \cite{prajna} this technique is extended to include systems 
with non-polynomial vector fields, which can be transformed to an equivalent system with polynomial
vector fields under equality and inequality constraints on the state variables. Unfortunately these techniques 
are limited to only certain class of autonomous systems. In general, it is numerically difficult 
to check the conditions to be satisfied by a Lyapunov candidate and a Lyapunov function. 
Checking positive definiteness of simple polynomial functions is also an NP-hard problem when 
polynomial has degree $4$ or higher \cite{murty}. \\
\indent  Checking the global asymptotic stability of an equilibrium point of a nonlinear autonomous system is a more difficult problem than checking stability or asymptotic stability of an equilibrium point. This is because while checking the global asymptotic stability, one cannot use the technique of local linearization about an equilibrium point. In this paper, we provide a necessary condition in terms of a local minima of the $h$ function (to be defined later in section-\ref{h-function}) that every Lyapunov function $V$, corresponding to the globally asymptotically stable equilibrium point of $\dot{\mathbf{x}}=f(\mathbf{x})$, must satisfy. As this necessary condition on a Lyapunov function does not involve system dynamics, it is easier 
to check. Moreover, since every Lyapunov function corresponding to the globally asymptotically stable equilibrium point of $\dot{\mathbf{x}}=f(\mathbf{x})$ has to satisfy this condition, the set of valid 
Lyapunov function candidates becomes much smaller. This would naturally benefit any method of finding 
a Lyapunov function corresponding to the globally asymptotically stable equilibrium point of $\dot{\mathbf{x}}=f(\mathbf{x})$. \\
\indent For the global asymptotic stability analysis of an equilibrium point, the necessary condition obtained from the Theorem \ref{chimu} can be used as a first level test for a Lyapunov function candidate. 
Note that, this first level test is based on the generalized steepest descent method given in section-\ref{huer}, and is easier to check than checking positive definiteness of a function.
Therefore, the computationally more intensive positive definiteness check can be spared for functions which fail to satisfy this first level test. 
The generalized steepest descent method given in section-\ref{huer} can also be used as a heuristic to check the local positive definiteness of a function, which is a well known necessary condition for a Lyapunov function corresponding to a stable and/or asymptotically stable equilibrium point of an autonomous system. \\
\indent While tackling the problem of finding a Lyapunov function to conclude the global asymptotic stability of the equilibrium point, one obvious start is by employing continuously differentiable and coercive functions which can be written as sum of squares. In order to conclude that a function in the above class is a Lyapunov function, one needs to check that the time derivative of this function along system trajectories is negative definite, which is again numerically very difficult. With our proposed necessary condition, the set of functions on which this negative definiteness condition (involving system dynamics) needs to be checked can be drastically reduced.\\
\indent This paper is organized as follows. In section-\ref{not-pre}, we explain notations 
and some preliminaries. Lyapunov theory and LaSalle's invariance principle are described in brief in section-\ref{lpv-theory}. In section-\ref{h-function}, we define the $h$ function which will be used later to state our necessary condition.
 We state and prove our main result
in section-\ref{main-result}, which gives a necessary condition for a Lyapunov function $V$ corresponding to the globally asymptotically stable equilibrium point.  In section-\ref{huer}, we explain the generalized steepest descent method to check the necessary condition obtained from Theorem \ref{chimu}. This method is useful for numerically ruling out some Lyapunov 
function candidates corresponding to the global asymptotic stability. We also give some examples to demonstrate usefulness of our necessary condition. Finally section-\ref{con-futrwrk} contains conclusions and future work. In appendix, we state and prove some auxiliary results related to the $h$ function.

\section{Notations and Preliminaries}\label{not-pre}
\indent $\mathbb{R}$ denotes the field of real numbers, and $\mathbb{R}^n$ is the $n$-dimensional real Euclidean space over $\mathbb{R}$. The set of natural numbers is denoted by $\mathbb{N}$. We use $\mathbb{R}_+$ to denote non-negative real numbers, and $\mathbb{R}_{++}$ to denote positive real numbers. Lowercase bold faced 
letters are used to denote vectors in 
$\mathbb{R}^n$; and lowercase non-bold faced letters to denote real scalars.\\
\indent $\parallel \cdot \parallel$ denotes the Euclidean norm or $2$-norm 
on $\mathbb{R}^n$. The ball and sphere in $\mathbb{R}^n$ of radius $r>0$ centered 
at $\mathbf{x}_0$, with respect to $2$-norm, are defined as:
\begin{eqnarray}
 B(\mathbf{x}_0,r)&:=&\{\mathbf{x} \in \mathbb{R}^n \mid \;\parallel \mathbf{x} - \mathbf{x}_0
 \parallel \,<\, r \}\,, \\
 S(\mathbf{x}_0,r)&:=&\{\mathbf{x} \in \mathbb{R}^n \mid \; \parallel \mathbf{x} - \mathbf{x}_0
 \parallel \,=\, r \}\,.
\end{eqnarray}
$\parallel \cdot \parallel_{\infty}$ denotes the $\infty$-norm on $\mathbb{R}^n$. One analogously defines the ball and sphere in $\mathbb{R}^n$ of radius $r$ centered at $\mathbf{x}_0$, with respect to $\infty$-norm. Let $D \subseteq \mathbb{R}^n$; we denote the interior and the closure of $D$ by $D^o$ and 
$\overline{D}$ respectively. The boundary of $D$, denoted by $\partial D$, is defined as 
$\partial D := \overline{D} \setminus D^o$. \\
\indent A function $g: \mathbb{R} \rightarrow \mathbb{R}$ is said to be 
increasing, if $x<y$ $\Longrightarrow$ $g(x)\leq g(y)$; and strictly increasing if $x<y$ 
$\Longrightarrow$ $g(x) < g(y)$. A real valued function $f:\mathbb{R}^n \rightarrow \mathbb{R}$ is said to be coercive, if for every sequence $\{\mathbf{x}_n\}\in \mathbb{R}^n$ which satisfy $\parallel \mathbf{x}_n \parallel\; \rightarrow\, \infty$, we have $\displaystyle \lim_{n \rightarrow \infty} f(\mathbf{x}_n)= \infty$ (see, \cite{bert}). Let $  f : D \rightarrow \mathbb{R} $, where $D 
\subseteq \mathbb{R}^n$; and let $E \subseteq D$. We denote the restriction of $f$ to $E$ as 
$f\vert_{E}$. For a function $f:\mathbb{R}^n \rightarrow \mathbb{R}$, we denote its 
gradient by $\nabla\, f(\cdot)$. \\
\indent A real valued function $ f: G \rightarrow \mathbb{R} $, where $G \subseteq \mathbb{R}^n$, is 
 said to be lower semi-continuous at a point $ \mathbf{x} \in G$, if for every sequence 
$ \{\mathbf{x}_k\} $ in $ G $ that converges to $ \mathbf{x} $, $ f(\mathbf{x}) \leq 
\displaystyle\liminf_{ k \rightarrow \infty} f(\mathbf{x}_k)$. A real valued function $f: G \rightarrow \mathbb{R} $ is said to be lower semi-continuous, if it is lower semi-continuous
at every $\mathbf{x} \in G$ (see, \cite{bert}). \\
\indent A function $f:G \rightarrow \mathbb{R}^n$, where $G \subseteq \mathbb{R}^n$, is said to be locally Lipschitz at a point $\mathbf{x}_0 \in G$, if there exists $\varepsilon \in \mathbb{R}_{++}$ and a Lipschitz constant $l \in \mathbb{R}_{++}$ such that, $\forall\; \mathbf{x},\mathbf{y} \in B(\mathbf{x}_0,\varepsilon)$ the following condition is satisfied:
 \begin{eqnarray}
  \parallel f(\mathbf{x}) - f(\mathbf{y}) \parallel \,&\leq&\; l \parallel \mathbf{x} - \mathbf{y} \parallel\;. 
 \end{eqnarray}
 A real valued function $f: G \rightarrow \mathbb{R}^n $ is said to be locally Lipschitz, if it is locally Lipschitz at every $\mathbf{x}_0 \in G \subseteq \mathbb{R}^n$ (see, \cite{khalil}). \\
\indent Consider the following equivalence relation on the vectors in $\mathbb{R}^n$: 
$\mathbf{d}_1$ is said to be equivalent to $\mathbf{d}_2$, denoted as 
$\mathbf{d}_1 \sim \mathbf{d}_2$, if $\mathbf{d}_1=\alpha \mathbf{d}_2$ for some $\alpha>0$. 
The set of equivalence classes of vectors in $\mathbb{R}^n$ induced by this relation are the directions in 
$\mathbb{R}^n$. Therefore, directions in $\mathbb{R}^n$ can be represented as points on the 
unit sphere $S(\mathbf{0},1)$. The induced topology on $S(\mathbf{0},1)$ from 
$\mathbb{R}^n$ is used to define the open and closed sets of $S(\mathbf{0},1)$ as follows. 
\begin{definition}
 Let $E \subseteq S(\mathbf{0},1)$. 
\begin{itemize}
 \item $\mathbf{d} \in E$ is said to be an interior point of $E$ (with respect to the induced 
 topology on $S(\mathbf{0},1)$), if $\exists$ $B(\mathbf{d}, \varepsilon)$ for some 
 $\varepsilon > 0$ such that $B(\mathbf{d}, \varepsilon) \cap S(\mathbf{0},1) \subset E$. 
 The set of all such interior points of $E$ is called the interior of $E$ with respect to the 
 induced topology on $S(\mathbf{0},1)$.
\item $E$ is said to be an open subset of $S(\mathbf{0},1)$, if every $\mathbf{d} \in E$ is 
an interior point of $E$ with respect to the induced topology on $S(\mathbf{0},1)$.
\item $E$ is said to be a closed subset of $S(\mathbf{0},1)$, if every limit point of $E$ 
belongs to $E$.
\item We define the boundary of $E$ (with respect to the induced topology on $S(\mathbf{0},1))$
as $\partial E := \overline{E} \setminus E^o$, where $E^o$ is with respect to the 
induced 
topology on $S(\mathbf{0},1)$.
\end{itemize}
\end{definition}
\indent For $\mathbf{d} \in S(\mathbf{0},1)$, we will use $N_{\varepsilon}(\mathbf{d})$ to denote its neighborhood on $S(\mathbf{0},1)$; i.e. $N_{\varepsilon}(\mathbf{d}) = B(\mathbf{d}, \varepsilon) \cap S(\mathbf{0},1)$.
\section{Lyapunov Theory and LaSalle's Invariance Principle}\label{lpv-theory}
\indent In this section we briefly cover Lyapunov theory and LaSalle's invariance principle. Reader can refer to \cite{khalil}, \cite{lyapunov}, \cite{vidya} for detailed treatment on these topics. \\
\indent Consider an autonomous system:
\begin{equation}\label{auto}
 \dot{\mathbf{x}} = f(\mathbf{x}),
\end{equation}
where $f: G \rightarrow \mathbb{R}^n$ is a locally Lipschitz map on its domain 
$G \subseteq \mathbb{R}^n$. A point $\mathbf{x}^* $ is said to be an equilibrium point of the 
autonomous system represented by \eqref{auto}, if it has the property that whenever the system
starts with the initial condition $\mathbf{x}(0)=\mathbf{x}^*$, it remains at $\mathbf{x}^*$
for all future time, i.e. $\mathbf{x}(t)= \mathbf{x}^*$, $\forall\, t \geq 0$. Therefore, 
$\mathbf{x}^*$ is an equilibrium point, if and only if it is a real root of the equation $f(\mathbf{x}) = \mathbf{0}$. 
\begin{definition}
 An equilibrium point $\mathbf{x} = \mathbf{x}^*$ of \eqref{auto} is said to be
\begin{itemize}
 \item stable if $\;\forall \; \varepsilon > 0$, $\exists\; \delta > 0$ such that, 
 $\mathbf{x}(0) \in B(\mathbf{x}^*, \delta) \Longrightarrow  \mathbf{x}(t) \in 
 B(\mathbf{x}^*,\varepsilon)$, $\forall\, t\geq 0 $.
 \item unstable if it is not stable.
 \item asymptotically stable if it is stable and $\delta$ can be chosen such that 
\[  \mathbf{x}(0) \in B(\mathbf{x}^*, \delta) \Longrightarrow \displaystyle\lim_{t \rightarrow 
\infty} \mathbf{x}(t) = \mathbf{x}^*.\]
\end{itemize}
\end{definition}
\indent Lyapunov theorem, which is stated below, gives a sufficient condition for 
stability and asymptotic stability of an equilibrium point $\mathbf{x}^*$.
\begin{theorem}\label{lya}
 Let $\mathbf{x}^*$ be an equilibrium point of \eqref{auto}. Let $ V: D \rightarrow 
 \mathbb{R} $, where $D \subseteq \mathbb{R}^n$ is an open set containing $\mathbf{x}^*$, be a continuously 
 differentiable function such that;
\begin{eqnarray}
& V(\mathbf{x}) > 0,\; \forall\; \mathbf{x} \in 
(D \setminus \{\mathbf{x}^*\}); \mbox{ and } V(\mathbf{x}^*) =0\,,  & \label{lya can} \\ \nonumber \\
&  \dot{V}(\mathbf{x}) = (\nabla V(\mathbf{x}))^Tf(\mathbf{x}) \leq 0, \; \forall \; 
\mathbf{x} \in D \,.& \label{lya fn} \,
\end{eqnarray}
Then, $\mathbf{x}^* $ is a stable equilibrium point \eqref{auto}. Moreover, if $ \dot V(\mathbf{x}) < 0, \; \forall \; \mathbf{x} \in (D \setminus \{\mathbf{x}^*\})$, 
then $\mathbf{x}^* $ is an asymptotically stable equilibrium point \eqref{auto}. \vspace{0.05in}
\end{theorem}\\
A continuously differentiable function $ V : D \rightarrow \mathbb{R} $ satisfying 
\eqref{lya can} is called a Lyapunov candidate; and a continuously differentiable function 
$ V : D \rightarrow \mathbb{R} $ satisfying both \eqref{lya can} and \eqref{lya fn} is called 
a Lyapunov function.\\
\indent Suppose $\mathbf{x}^*$ is an asymptotically stable equilibrium point of \eqref{auto}. 
Then, the largest region around $\mathbf{x}^*$ which satisfies the property that, any trajectory 
starting in that region will converge to $\mathbf{x}^*$ (as $t \rightarrow \infty$) is called the region of attraction of 
$\mathbf{x}^*$. 
\begin{definition}
 If the region of attraction for an asymptotically stable equilibrium point 
$\mathbf{x}^*$ is entire $\mathbb{R}^n$, then $\mathbf{x}^*$ is called the globally asymptotically stable equilibrium point of \eqref{auto}.
\end{definition}\\
Clearly if $\mathbf{x}^*$ is a globally asymptotically 
stable equilibrium point of \eqref{auto}, then it must be the unique equilibrium point of 
\eqref{auto}. The Barbashin-Krasovskii theorem, stated below, gives a sufficient condition 
for $\mathbf{x}^*$ to be the globally asymptotically stable equilibrium point. 
\begin{theorem}\label{krasov}
 Let $\mathbf{x}^*$ be an equilibrium point of \eqref{auto}. Let $V:\mathbb{R}^n \rightarrow 
 \mathbb{R}$ be a continuously differentiable function such that:
\begin{eqnarray}
 &V(\mathbf{x}) > 0,\;\, \forall \, \mathbf{x} \neq \mathbf{x}^*; 
 \mbox{  and  } V(\mathbf{x}^*) = 0\,,&  \label{pos-def}\\ 
 & \parallel \mathbf{z} \parallel \rightarrow \infty \;\, \Longrightarrow \;\, V(\mathbf{x}^* 
 + \mathbf{z}) \rightarrow \infty\,, & \label{crcv}\\ 
 & \dot{V}(\mathbf{x})= (\nabla V(\mathbf{x}))^Tf(\mathbf{x}) < 0, \;\, \forall \, \mathbf{x} 
 \neq \mathbf{x}^*\, . & \label{dynamics}
\end{eqnarray}
Then, $\mathbf{x}^*$ is the globally asymptotically stable equilibrium point of 
\eqref{auto}. \vspace{0.05in}
\end{theorem} \\
With respect to the global asymptotic stability, a continuously differentiable function $ V : D \rightarrow \mathbb{R} $ satisfying 
\eqref{pos-def} and \eqref{crcv} is called a Lyapunov candidate. Whereas, a continuously differentiable function 
$ V : D \rightarrow \mathbb{R} $ satisfying all conditions in the Theorem-\ref{krasov} is called a Lyapunov function. \\
\indent We discuss below LaSalle's invariance principle. We state two special cases or corollaries of LaSalle's invariance principle; reader can refer to \cite{khalil} and \cite{vidya} for the general statement of LaSalle's invariance principle.
\begin{proposition}
 Let $\mathbf{x}^*$ be an equilibrium point of \eqref{auto}. Let $ V: D \rightarrow 
 \mathbb{R} $, where $D \subseteq \mathbb{R}^n$ is an open set containing $\mathbf{x}^*$, be a continuously 
 differentiable function satisfying \eqref{lya can} and \eqref{lya fn}. Let $S:= \{\mathbf{x} \in D \mid \dot{V}(\mathbf{x})=0\}$ and suppose that no solution can stay identically in $S$, other than the trivial solution $\mathbf{x}(t)\equiv\mathbf{x}^*$. Then, $\mathbf{x}^*$ is an asymptotically stable equilibrium point of \eqref{auto}.
\end{proposition}
\begin{proposition}
 Let $\mathbf{x}^*$ be an equilibrium point of \eqref{auto}. Let $V:\mathbb{R}^n \rightarrow 
 \mathbb{R}$ be a continuously differentiable function satisfying \eqref{pos-def}, \eqref{crcv}, and $\dot{V}(\mathbf{x}) \leq 0$ for all $\mathbf{x} \in \mathbb{R}^n$. Let $S:= \{\mathbf{x} \in \mathbb{R}^n \mid \dot{V}(\mathbf{x})=0\}$ and suppose that no solution can stay identically in $S$, other than the trivial solution $\mathbf{x}(t)\equiv\mathbf{x}^*$. Then, $\mathbf{x}^*$ is the globally asymptotically stable equilibrium point of \eqref{auto}.
\end{proposition}\\
In general, finding the set $S$ numerically is a daunting task. Only in few simple cases like inverted pendulum with the energy function as a Lyapunov candidate $V$, the set $S$ can be easily found. Therefore, though theoretically these special cases of LaSalle's invariance principle are very interesting results, it is extremely difficult to apply them in practice.

\section{h function}\label{h-function}
\indent Consider a Lyapunov candidate $V$ corresponding to the globally asymptotically stable equilibrium point $\mathbf{x}^*$ of \eqref{auto}. Let $\mathbf{d} \in S(\mathbf{0},1)$, now  $V\vert_{\{\mathbf{x}^*+\gamma\mathbf{d}\mid \gamma \geq 0\}}$ is a function of one variable $\gamma$. Let us denote this one variable function by $k_d:\mathbb{R}_+ \rightarrow \mathbb{R}$, i.e.
\begin{equation}
 k_d(\gamma) := V(\mathbf{x}^* + \gamma \mathbf{d})\,.
\end{equation}
We define the function $h: S(\mathbf{0},1) \rightarrow \mathbb{R}_{++} \cup \{\infty\}$ as follows:
\begin{equation}
 h(\mathbf{d}) := \mbox{ minimum } \gamma \in \mathbb{R}_{++} \mbox{ which satisfies } k_d'(\gamma) = 0\,.
\end{equation}
If such a $\gamma$ does not exist, i.e. if $V\vert_{\{\mathbf{x}^*+\gamma\mathbf{d}\mid \gamma \geq 0\}}$ is a strictly increasing function without an inflection point, then we declare $h(\mathbf{d})=\infty$. \\
\indent For every direction point $\mathbf{d} \in S(\mathbf{0},1)$ for which $h(\mathbf{d}) < \infty$, we define the corresponding point $\mathbf{z}_d$ as follows:
\begin{equation}\label{zd}
 \mathbf{z}_d := \mathbf{x}^* + h(\mathbf{d})\,\mathbf{d}\,.
\end{equation}
We will use the above definition in next section to state our main result. From the above equation it is clear that, $\parallel \mathbf{z}_d - \mathbf{x}^* \parallel \,=\, h(\mathbf{d})$. It is apparent from the definition of functions $k_d$ and $h$ that, for every $\mathbf{d} \in S(\mathbf{0},1)$ for which $h(\mathbf{d}) < \infty$, we have $\langle \nabla V(\mathbf{z}_d),\, \mathbf{d} \rangle = 0$; but note that, $\nabla V(\mathbf{z}_d)$ need not be zero (refer Figure \ref{pr}).
\begin{figure}[ht]
  \begin{center}
    \psfrag{x^*}{{\scriptsize \textcolor{blue}{$\mathbf{x}^*$}}}
    \psfrag{u}{{\scriptsize \textcolor{blue}{$\mathbf{d}$}}}
    \psfrag{span{u}}{{\scriptsize \textcolor{blue}{Span$\{\mathbf{d}\}$}}}
    \psfrag{0}{{\scriptsize \textcolor{blue}{$\mathbf{0}$}}}
    \psfrag{z_u=x^*+alp_u u}{{\scriptsize \textcolor{blue}{$\mathbf{z_d}=\mathbf{x}^*+h(\mathbf{d})\, \mathbf{d}$}}}
    \psfrag{g_v}{{\scriptsize \textcolor{blue}{$\nabla V(\mathbf{z}_d)$}}}
    \includegraphics[scale=0.42]{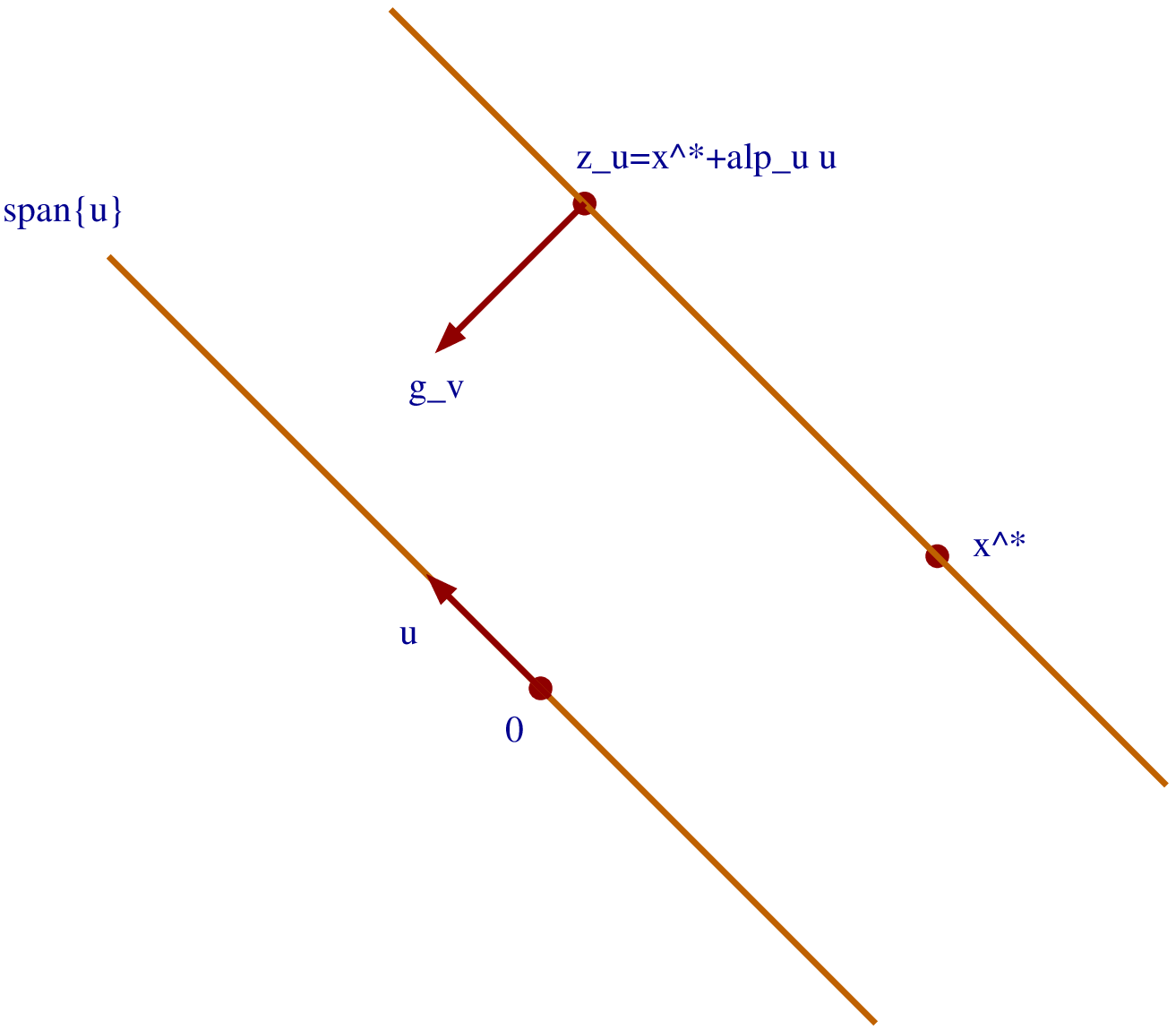}
    \caption{ }
    \label{pr}
    \end{center}
\end{figure}

\section{Necessary Condition on a Lyapunov Function Corresponding to the Globally Asymptotically Stable Equilibrium Point}\label{main-result} 
\indent  In this section, we propose a new necessary condition on a Lyapunov function corresponding to the globally asymptotically stable equilibrium point. This new necessary condition is numerically easier to check compare to the known necessary conditions given in \eqref{pos-def} and \eqref{crcv}.\\
\indent We give below a sufficient condition under which there exists a point $\mathbf{z} \neq \mathbf{x}^*$ such that $\nabla V(\mathbf{z})=0$, where $V: \mathbb{R}^n \rightarrow \mathbb{R}$ is a Lyapunov candidate corresponding to the globally asymptotically stable equilibrium point $\mathbf{x}^*$. If such a point exists, then $V$ will not satisfy the condition \eqref{dynamics} in Theorem \ref{krasov}; and hence it cannot be a Lyapunov function. The sufficient condition for existence of a point $\mathbf{z} \neq \mathbf{x}^*$, such that $\nabla V(\mathbf{z})=0$, gives a necessary condition for $V: \mathbb{R}^n \rightarrow \mathbb{R}$ to be a Lyapunov function corresponding to the globally asymptotically stable equilibrium point $\mathbf{x}^*$.
\begin{theorem}\label{chimu}
 Suppose  $V:\mathbb{R}^n \rightarrow \mathbb{R}$ is a Lyapunov candidate corresponding to the globally asymptotically stable equilibrium point $\mathbf{x}^*$ of \eqref{auto}, which satisfies the following: 
 \begin{equation}
  \inf\, \{h(\mathbf{d}) \mid \mathbf{d} \in S(\mathbf{0},1)\} < \infty \,.
 \end{equation}
 Then for any local minimizer $\mathbf{w} \in S(\mathbf{0},1)$ of $h$ with $h(\mathbf{w}) < \infty$, we have $\nabla V(\mathbf{z}_w)=0$.
\end{theorem}
 \begin{proof}
 Let $\mathbf{w} \in S(\mathbf{0},1)$ be a local minimizer of $h$ with $h(\mathbf{w}) < \infty$. Then with respect to induced topology on $S(\mathbf{0},1)$, there exists a neighborhood $N_{\varepsilon}(\mathbf{w}) := B(\mathbf{w}, \varepsilon) \cap S(\mathbf{0},1)$ of $\mathbf{w}$ for some $\varepsilon > 0$ such that:
 \begin{equation}\label{nrst}
 h(\mathbf{w}) \; \leq \; h(\mathbf{u})\,, \; \forall\;\mathbf{u} \in N_{\varepsilon}(\mathbf{w})\,.
\end{equation}
\indent Suppose $\nabla V(\mathbf{z}_w) \neq \mathbf{0}$. As $h(\mathbf{w}) < \infty$, we have $\langle \nabla V(\mathbf{z}_w),\, \mathbf{w} \rangle = 0$. Let us define some open half-spaces which we need later:
\begin{eqnarray}
 \mathcal{H}_1 &:=& \{\mathbf{x} \in \mathbb{R}^n \mid \langle \nabla V(\mathbf{z}_w),\, \mathbf{x} \rangle \,<\, 0\}\,,  \\
 \mathcal{H}_2 &:=& \{\mathbf{x} \in \mathbb{R}^n \mid \langle \mathbf{w},\, \mathbf{x} \rangle \,\leq \, 0\}\,. 
 \end{eqnarray}
 Directional derivative of $V$ at $\mathbf{z}_w$ in every direction $\mathbf{u} \in (\mathcal{H}_1 \cap \mathcal{H}_2)$ is negative. 
 \begin{figure}[ht]
  \begin{center}
    \psfrag{x^*}{{\scriptsize \textcolor{blue}{$\mathbf{x}^*$}}}
    \psfrag{z_w}{{\scriptsize \textcolor{blue}{$\mathbf{z}_w$}}}
    \psfrag{w}{{\scriptsize \textcolor{blue}{$\mathbf{w}$}}}
    \psfrag{0}{{\scriptsize \textcolor{blue}{$\mathbf{0}$}}}
    \psfrag{grad V(z_w)}{{\scriptsize \textcolor{blue}{$\nabla V(\mathbf{z}_w)$}}}
    \psfrag{d_o}{{\scriptsize \textcolor{blue}{$\mathbf{d}_{\theta}$}}} 
    \psfrag{th}{{\scriptsize \textcolor{blue}{$\theta$}}}
   \includegraphics[scale=0.41]{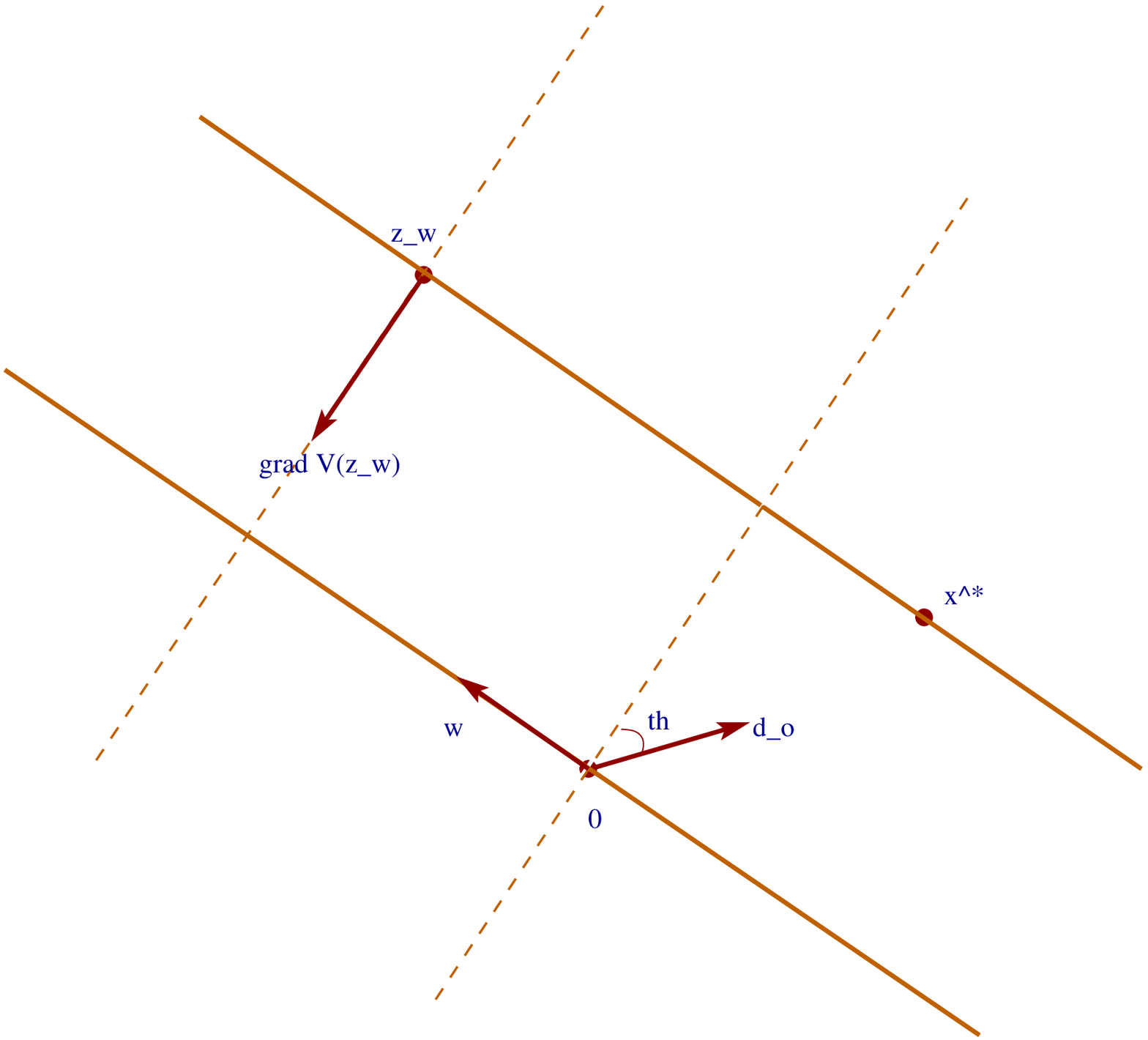}
   \caption{ }
   \label{mprf}
  \end{center}
\end{figure}
Consider a two dimensional subspace $\mathbb{V} := \mbox{Span} \{\nabla V(\mathbf{z}_w), \mathbf{w}\}$. Every direction $\mathbf{d}_{\theta} \in (\mathcal{H}_1 \cap \mathcal{H}_2 \cap \mathbb{V})$, with $\parallel \mathbf{d}_{\theta} \parallel \,=1$, can be parametrized by an angle $\theta$ ($0 \leq \theta < \pi/2$) it makes with $-\nabla V(\mathbf{z}_w)$ (refer Figure \ref{mprf}). $\langle \nabla V(\mathbf{z}_w),\, \mathbf{d}_{\theta} \rangle$ is negative and strictly increasing for $\theta \in [0,\, \pi/2)$. 
Therefore, as $V$ is a continuously differentiable function, $\exists\; \alpha >0$ such that for every $\beta \in (0,\,\alpha]$ following holds:\footnote{It follows from the Taylor series expansion and mean value theorem.} 
\begin{equation}\label{arc}
  V(\mathbf{z}_w + \beta \mathbf{d}_{\theta}) \mbox{ increases, as } \theta \mbox{ increases in } [0,\,\pi/2)\,.
\end{equation}
\indent Now consider a sequence of directions $(\mathbf{u}_n) \in \mathcal{H}_1 \cap (-\mathcal{H}_2) \cap \mathbb{V}$, with $\parallel \mathbf{u}_n \parallel = 1$, converging to $\mathbf{w}$. As $\displaystyle \lim_{n \rightarrow \infty} \mathbf{u}_n = \mathbf{w}$, there exists $n_o \in \mathbb{N}$ and $\gamma >0$ such that $\forall\; n \geq n_o$, $(\mathbf{x}^* + \gamma \mathbf{u}_n) \in B(\mathbf{z}_w, \alpha) \cap \mathbb{V}$ (refer Figure \ref{mprf1}).\footnote{Note that $S(\mathbf{z}_w,\alpha) \cap \mathbb{V}$ and $S(\mathbf{z}_w,\beta) \cap \mathbb{V}$ would be circles; but in Figure \ref{mprf1} and Figure \ref{mprf2}, we have shown only the arcs of these circles we are interested in.}
\begin{figure}[ht]
  \begin{center}
    \psfrag{x^*}{{\scriptsize \textcolor{blue}{$\mathbf{x}^*$}}}
    \psfrag{z_w}{{\scriptsize \textcolor{blue}{$\mathbf{z}_w$}}}
    \psfrag{w}{{\scriptsize \textcolor{blue}{$\mathbf{w}$}}}
    \psfrag{0}{{\scriptsize \textcolor{blue}{$\mathbf{0}$}}}
    \psfrag{grad V(z_w)}{{\scriptsize \textcolor{blue}{$\nabla V(\mathbf{z}_w)$}}}
    \psfrag{S(z_w,a)}{{\scriptsize \textcolor{blue}{$S(\mathbf{z}_w,\alpha) \cap \mathbb{V}$}}} 
    \psfrag{S(z_w,b)}{{\scriptsize \textcolor{blue}{$S(\mathbf{z}_w,\beta) \cap \mathbb{V}$}}}
    \psfrag{u_n}{{\scriptsize \textcolor{blue}{$\mathbf{u}_n\vert_ {n \geq n_o}$}}}
    \psfrag{x^*+gu_n}{{\scriptsize \textcolor{blue}{$(\mathbf{x}^*+\gamma \mathbf{u}_n)\vert_{n \geq n_o}$}}}
    \includegraphics[scale=0.45]{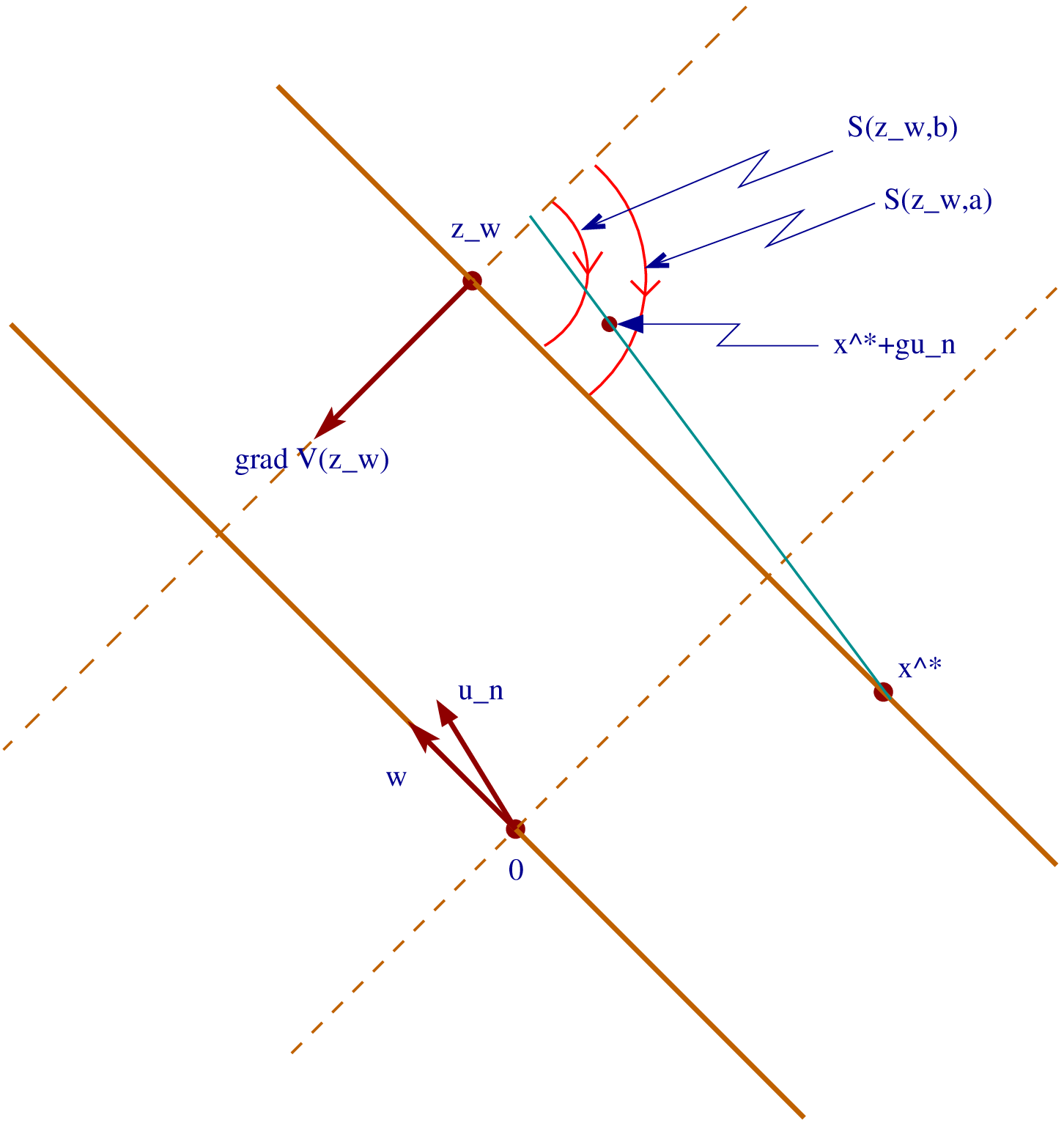}
    \caption{ } 
    \label{mprf1}
  \end{center}
\end{figure}
Consider an arbitrary direction $\mathbf{v} \in \{\mathbf{u}_n \mid n \geq n_o\}$. Let $\mathbf{y}_v$ be the point of intersection of $S(\mathbf{x}^*, h(\mathbf{w}))$ with the ray $\{\mathbf{x}^*+\gamma\,\mathbf{v} \;\vert\; \gamma > 0\}$; as shown in Figure \ref{mprf2}. 
\begin{figure}[ht]
  \begin{center}
    \psfrag{x^*}{{\scriptsize \textcolor{blue}{$\mathbf{x}^*$}}}
    \psfrag{z_w}{{\scriptsize \textcolor{blue}{$\mathbf{z}_w$}}}
    \psfrag{w}{{\scriptsize \textcolor{blue}{$\mathbf{w}$}}}
    \psfrag{0}{{\scriptsize \textcolor{blue}{$\mathbf{0}$}}}
    \psfrag{grad V(z_w)}{{\scriptsize \textcolor{blue}{$\nabla V(\mathbf{z}_w)$}}}
    \psfrag{S(x^*, z_w-x^*)}{{\scriptsize \textcolor{blue}{$S(\mathbf{x}^*, h(\mathbf{w})) \cap \mathbb{V}$}}}
    \psfrag{S(z_w,a)}{{\scriptsize \textcolor{blue}{$S(\mathbf{z}_w,\alpha) \cap \mathbb{V}$}}}
    \psfrag{v}{{\scriptsize \textcolor{blue}{$\mathbf{v}$}}}
    \psfrag{y_v}{{\scriptsize \textcolor{blue}{$\mathbf{y}_v$}}}
    \psfrag{x_v}{{\scriptsize \textcolor{blue}{$\mathbf{x}_v$}}}
    \includegraphics[scale=0.45]{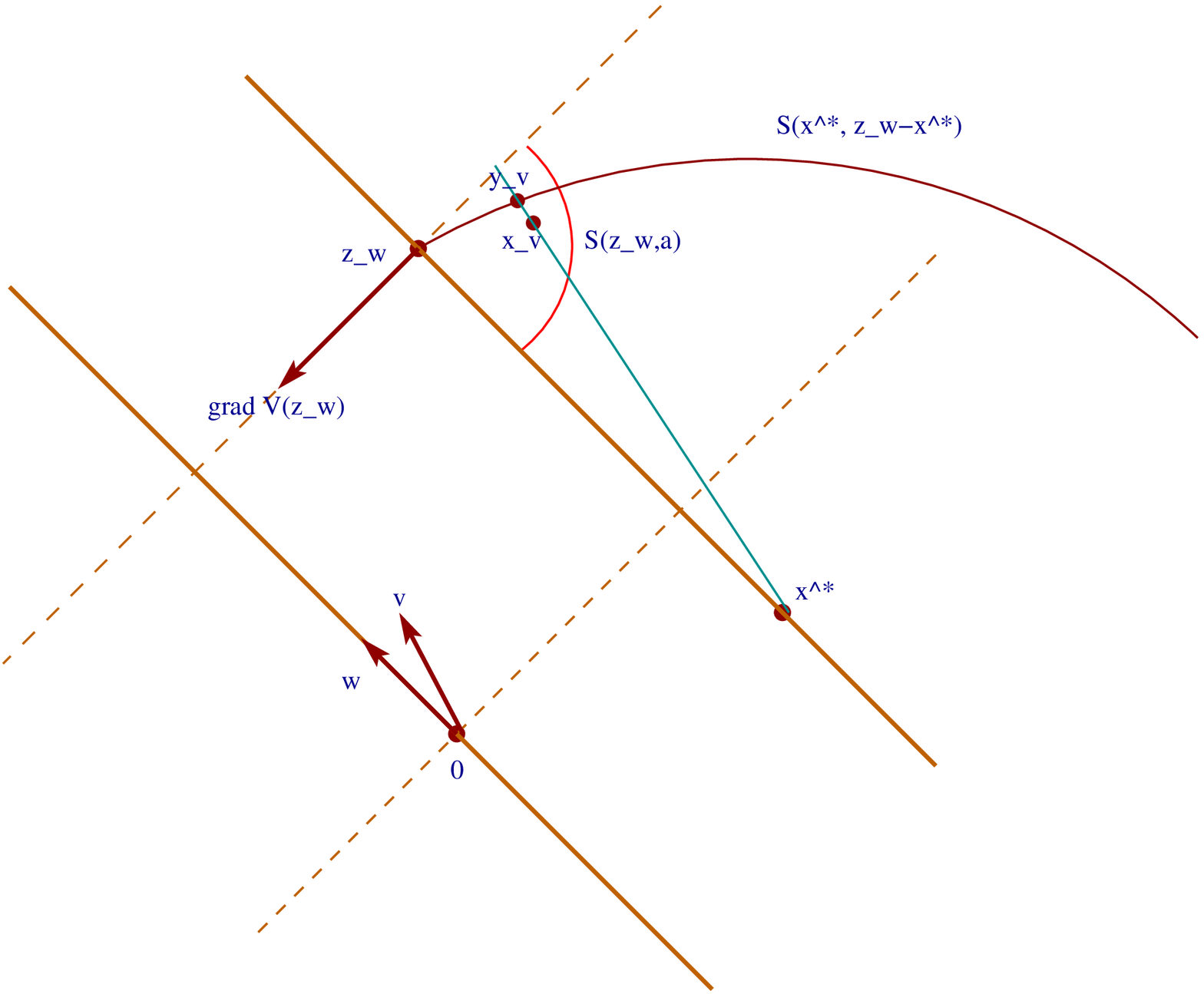}
    \caption{ }
    \label{mprf2}
  \end{center}
\end{figure} \\
\indent Now, consider a triangle in the plane $\mathbb{V}$ with vertices $\mathbf{z}_w,\, \mathbf{y}_v$, and $\mathbf{x}^*$. As shown in Figure \ref{circle}, angle at the vertex $\mathbf{y}_v$ in this triangle is outside the smaller semi-circle with the line segment $[\mathbf{x}^*,\, \mathbf{z}_w]$ as a diameter. 
Therefore, angle at the vertex $\mathbf{y}_v$ is less than $\pi/2$ and $\exists \; \mathbf{x}_v \in (\mathbf{x}^*,\,\mathbf{y}_v)$ such that, $\parallel \mathbf{x}_v - \mathbf{z}_w \parallel \;=\; \parallel \mathbf{y}_v - \mathbf{z}_w \parallel$. However, from \eqref{arc} we have, $V(\mathbf{x}_v) \geq V(\mathbf{y}_v)$.
\begin{figure}[ht]
  \begin{center}
    \psfrag{x^*}{{\scriptsize \textcolor{blue}{$\mathbf{x}^*$}}}
    \psfrag{z_w}{{\scriptsize \textcolor{blue}{$\mathbf{z}_w$}}}
    \psfrag{S(x^*, z_w-x^*)}{{\scriptsize \textcolor{blue}{$S(\mathbf{x}^*,h(\mathbf{w})) \cap \mathbb{V}$}}}
    \psfrag{y_v}{{\scriptsize \textcolor{blue}{$\mathbf{y}_v$}}}
    \psfrag{x_v}{{\scriptsize \textcolor{blue}{$\mathbf{x}_v$}}}
    \psfrag{S(x^*+z_w/2,z_w-x^*/2)}{{\tiny \textcolor{blue}{$S(\frac{\mathbf{x}^*+\mathbf{z}_w}{2},\frac{h(\mathbf{w})}{2}) \cap \mathbb{V}$}}}
    \includegraphics[scale=0.45]{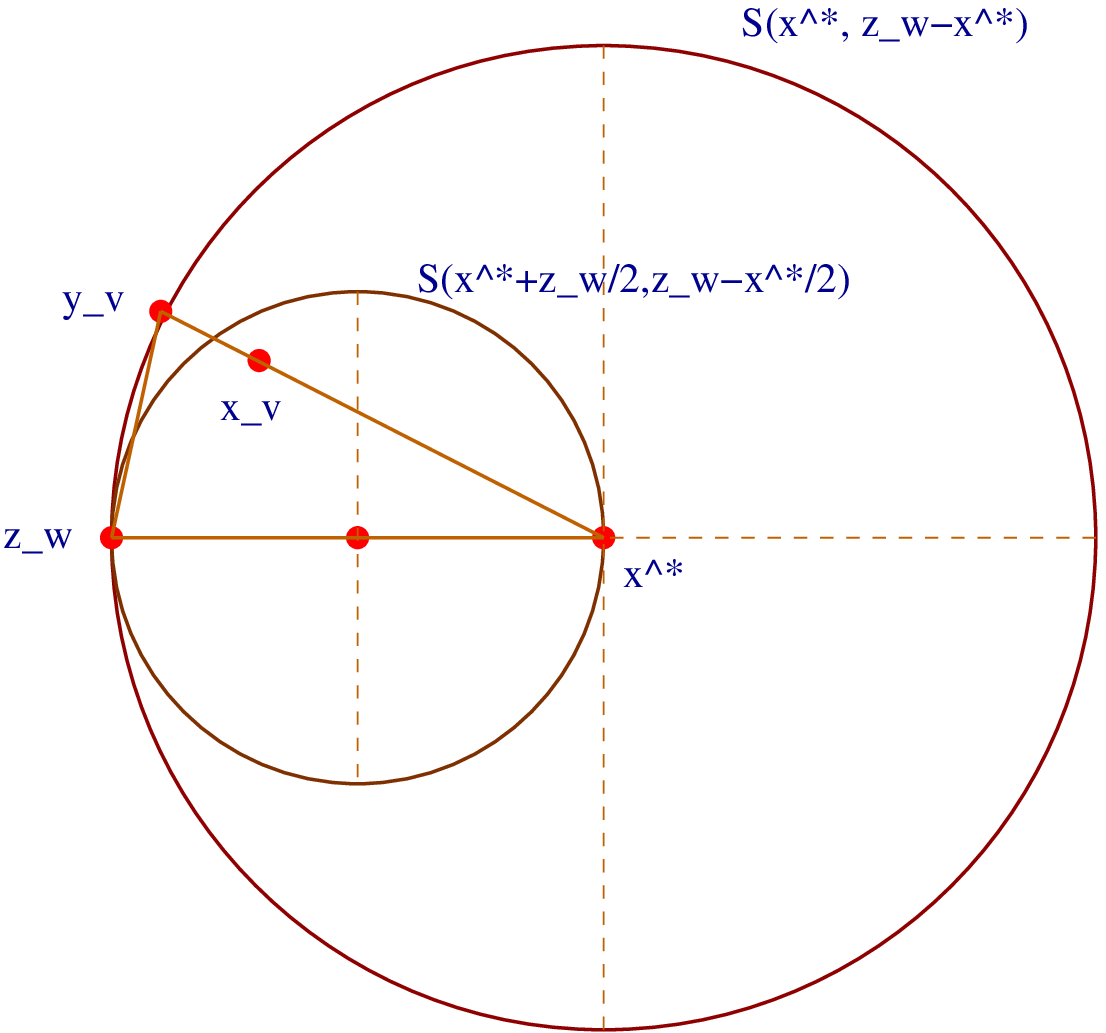}
    \caption{ }
  \label{circle}
  \end{center}
\end{figure}\\
\indent As $V$ is a Lyapunov candidate, $V$ satisfies \eqref{pos-def}; and hence $V$ is strictly increasing in every direction at $\mathbf{x}^*$. However as $V(\mathbf{x}_v) \geq V(\mathbf{y}_v)$, we can say that $V\vert_{\{\mathbf{x}^*+\gamma\mathbf{v}\mid \gamma \geq 0\}}$ is not a strictly increasing function, and this one variable function has a local maxima at certain point $\tilde{\mathbf{z}}_v \in [\mathbf{x}_v,\,\mathbf{y}_v)$. As $V\vert_{\{\mathbf{x}^*+\gamma\mathbf{v}\mid \gamma \geq 0\}}$ has a local maxima at $\tilde{\mathbf{z}}_v$, we have $k_v'(\parallel \tilde{\mathbf{z}}_v - \mathbf{x}^* \parallel)\,=0\,$.
\begin{equation}\label{cntr}
 \therefore \; h(\mathbf{v})\;=\;\parallel \mathbf{z}_v - \mathbf{x}^* \parallel \; \leq \; \parallel \tilde{\mathbf{z}}_v - \mathbf{x}^* \parallel \; < \;\parallel \mathbf{y}_v - \mathbf{x}^* \parallel \; = \; \parallel \mathbf{z}_w - \mathbf{x}^* \parallel\;=\;h(\mathbf{w})
\end{equation}
Therefore for every $\mathbf{v} \in \{\mathbf{u}_n \mid n \geq n_o\}$, we have $h(\mathbf{v})<h(\mathbf{w})$. This is a contradiction to the fact that $\mathbf{w}$ is a local minimizer of $h$. Therefore, $\nabla V(\mathbf{z}_w)=0$.
\end{proof} \vspace{0.1in} \\
{\large \textit{Remarks}}:
\begin{enumerate}
 \item From the above theorem, we get the following necessary condition for a Lyapunov candidate to be a Lyapunov function corresponding to the globally asymptotically stable equilibrium point: The function $h$ should not have any local minimizer with finite local minimum.\footnote{This necessary condition is without involving system dynamics. In \cite{self}, for a special autonomous system $\dot{\mathbf{x}}=-\mathbf{x}$, we have given a necessary and sufficient condition (without involving system dynamics) for a continuously differentiable function to be a Lyapunov function to conclude the global asymptotic stability of the origin. However, in general one can not find a necessary and sufficient condition on Lyapunov candidates to be a Lyapunov function, without involving dynamics of the system. This is because stability, asymptotic stability, or global asymptotic stability is an intrinsic property of an equilibrium point of the autonomous system.} 
 \item The more stronger and very obvious necessary condition would be: $\nabla V(\mathbf{x}) \neq 0$, $\forall\; \mathbf{x} \neq \mathbf{x}^*$; but there is no systematic method to check this condition numerically in a general case.
 \item If for every $\mathbf{d} \in S(\mathbf{0},1)$ the one variable function $V\vert_{\{\mathbf{x}^*+\gamma\mathbf{d}\mid \gamma \geq 0\}}$ is strictly increasing without an inflection point, i.e. $\inf\, \{h(\mathbf{d}) \mid \mathbf{d} \in S(\mathbf{0},1)\} = \infty $; then $h$ will not have any local minimizer with finite local minimum. 
 \item If $\inf\, \{h(\mathbf{d}) \mid \mathbf{d} \in S(\mathbf{0},1)\}=0$, then the global infimum of $h$ will not be attained, because range of $h$ is $(0,\infty]$. 
 \item For $\dot{x}=f(x)$, $x \in \mathbb{R}$; Lyapunov function is a one variable scalar valued function $V: \mathbb{R} \rightarrow \mathbb{R}$. In this case, there are only two directions and necessary condition reduces to the following: $V$ restricted to both directions must be a strictly increasing function without an inflection point.
 \item The above theorem has an obvious analogous counterpart for a Lyapunov candidate $V:D \rightarrow \mathbb{R}$, where $D \subseteq \mathbb{R}^n$, corresponding to an asymptotically stable equilibrium point $\mathbf{x}^*$ of \eqref{auto}, with appropriate changes in the definition of functions $k_d$ and $h$ according to the domain $D$ of $V$. This counterpart would give a necessary condition just like Remark-1 for Lyapunov functions corresponding to an asymptotically stable equilibrium point.
\end{enumerate}

\section{Generalized Steepest Descent Method}\label{huer}
\indent In this section, we give a method to find a local minimizer of the function $h$. This method is based on some of the ideas in the proof of Theorem \ref{chimu} and the steepest descent method (reader can refer to \cite{shetty}, \cite{bert}, \cite{ped} for steepest descent method).\\
\indent Suppose we want to check a continuously differentiable function $V:\mathbb{R}^n \rightarrow \mathbb{R}$ for being a Lyapunov function candidate corresponding to the globally asymptotically stable equilibrium point $\mathbf{x}^*$. In the generalized steepest descent method, we search for a point $\mathbf{z}\neq\mathbf{x}^*$ at which gradient of $V$ vanishes. If there exists such a point, then $V$ cannot be a Lyapunov function to conclude the global asymptotic stability of $\mathbf{x}^*$. Therefore it can be used as a first level test, by which checking positive definiteness of many functions can be avoided if you succeed in finding such a point. \\
\indent It is quite likely that while tackling the problem of finding a Lyapunov function to conclude the global asymptotic stability of the equilibrium point, one would start with continuously differentiable and coercive functions which can be written as sum of squares; and hence are known to be positive definite. In order to conclude that a function in the above class is a Lyapunov function, one needs to check that the time derivative of this function along system trajectories is negative definite, which is again numerically very difficult. If there exists a point $\mathbf{z}\neq\mathbf{x}^*$ at which the gradient of such function vanishes, then it cannot be a Lyapunov function to conclude the global asymptotic stability. Therefore in such cases also our method would be useful; because with this method the set of functions on which the negative definiteness condition, involving system dynamics, needs to be checked can be made much smaller. \\
\indent Consider an autonomous system given in \eqref{auto}. Suppose we are interested in checking the global asymptotic stability of an equilibrium point $\mathbf{x}^*$ of this autonomous system. Let $V:\mathbb{R}^n \rightarrow \mathbb{R}$ be a continuously differentiable function which we want to check for being a Lyapunov function candidate.
Now given $\mathbf{d} \in S(\mathbf{0},1)$, we can numerically find $h(\mathbf{d})$ by differentiating the one variable function $k_d(\gamma) := V(\mathbf{x}^* + \gamma \mathbf{d})$. By definition of $h$ function, $h(\mathbf{d})$ is nothing but smallest $\gamma >0$ for which derivative of $k_d$ becomes zero. Though we can numerically find $h(\mathbf{d})$ for given $\mathbf{d} \in S(\mathbf{0},1)$, we do not know the analytical expression of the function $h:S(\mathbf{0},1) \rightarrow \mathbb{R}_{++} \cup \{\infty\}$. If analytical expression of the function $h$ was known, then we could have searched for a local minimizer $\mathbf{w} \in S(\mathbf{0},1)$ of $h$ by either using the steepest descent method or some other optimization algorithm. If the function $h$ has a local minimizer $\mathbf{w}$ with $h(\mathbf{w}) < \infty$, then by the Theorem \ref{chimu} we know that: $\nabla V(\mathbf{z}_{w})=0$, where $\mathbf{z}_{w}=\mathbf{x}^*+h(\mathbf{w})\,\mathbf{w}$. As we do not know the analytical expression for the 
function $h$, we search 
for a local minimizer direction point $\mathbf{w} \in S(\mathbf{0},1)$ by what we call a 'generalized steepest descent method'.\\
\indent In order to explain the generalized steepest descent method, we will assume that we have a direction point $\mathbf{d} \in S(\mathbf{0},1)$ for which $h(\mathbf{d}) < \infty$. Later we will explain, how one could systematically search for such a direction point. If we have a direction $\mathbf{d} \in S(\mathbf{0},1)$ such that $h(\mathbf{d})\,<\,\infty$, then we can use the following generalized steepest descent method with non-exact line search to find a local minimizer $\mathbf{w}$ of $h$.
\begin{enumerate}
 \item Let $\mathbf{d}_0= \mathbf{d}$. For direction $\mathbf{d}_0$, calculate its corresponding point $\mathbf{z}_{d_0} = \mathbf{x}^* + h(\mathbf{d_0})\, \mathbf{d}_0$.
 Find $\nabla V(\mathbf{z}_{d_0})$. If $\nabla V(\mathbf{z}_{d_0})=\mathbf{0}$, $V$ cannot 
 be a Lyapunov function. If it is non-zero, then proceed as follows.
 \item Consider points of the form $(\mathbf{z}_{d_0} - \beta \nabla V(\mathbf{z}_{d_0}))$, $\beta > 0$ as shown in Figure \ref{nhrst}. 
 \begin{figure}[ht]
  \begin{center}
    \psfrag{x^*}{{\scriptsize \textcolor{blue}{$\mathbf{x}^*$}}}
    \psfrag{z_w}{{\scriptsize \textcolor{blue}{$\mathbf{z}_{d_0}$}}}
    \psfrag{w}{{\scriptsize \textcolor{blue}{$\mathbf{d}_0$}}}
    \psfrag{0}{{\scriptsize \textcolor{blue}{$\mathbf{0}$}}}
    \psfrag{grad V(z_w)}{{\scriptsize \textcolor{blue}{$\nabla V(\mathbf{z}_{d_0})$}}}
    \psfrag{v}{{\scriptsize \textcolor{blue}{$\mathbf{u}_{\beta}$}}}
    \psfrag{y_v}{{\scriptsize \textcolor{blue}{$\mathbf{z}_{d_0} - \beta \nabla V(\mathbf{z}_{d_0}) $}}}
    \psfrag{x_v}{{\scriptsize \textcolor{blue}{$\mathbf{z}_{u_{\beta}}$}}}
    \includegraphics[scale=0.45]{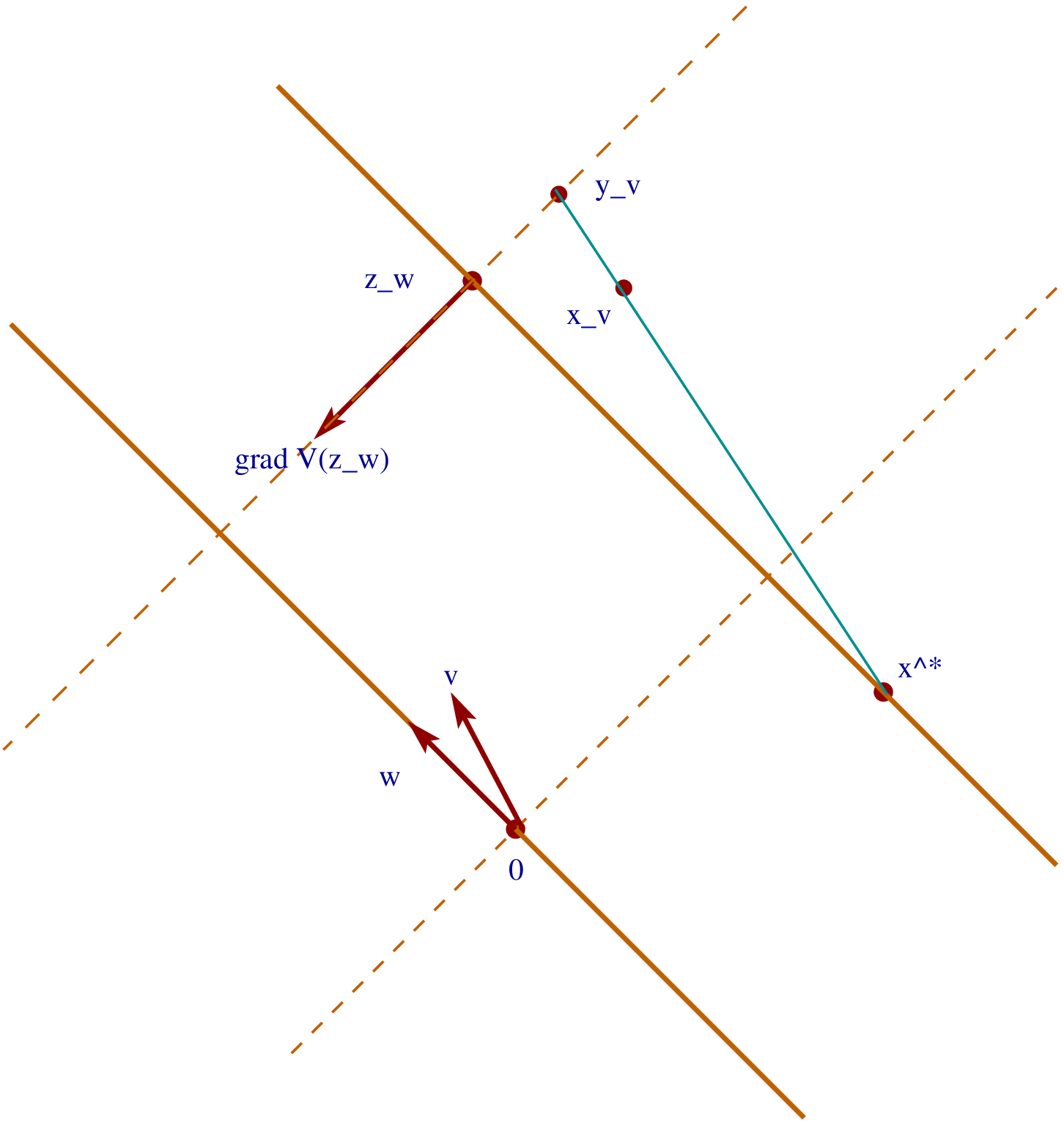}
    \caption{ }
    \label{nhrst}
  \end{center}
\end{figure}
 Corresponding to every point of this form, we get a unique direction $\mathbf{u}_{\beta} := (\mathbf{z}_{d_0} - \beta \nabla V(\mathbf{z}_{d_0})-\mathbf{x}^*)/\parallel \mathbf{z}_{d_0} - \beta \nabla V(\mathbf{z}_{d_0})-\mathbf{x}^* \parallel$. If $\mathbf{u}_{\beta}$ is sufficiently close to $\mathbf{d}_0$, i.e. if $\beta$ is sufficiently small, then $h(\mathbf{u}_{\beta}) < h(\mathbf{d}_0)$.\footnote{This has been shown in the proof of Theorem \ref{chimu}.}
 \item Evaluate $h$ at different directions $\mathbf{u}_{\beta}$ corresponding to points on the ray $\{\mathbf{z}_{d_0} - \beta \nabla V(\mathbf{z}_{d_0}) \mid \beta > 0\}$. A local minimizer of $h$ restricted to such points $\mathbf{u}_{\beta}$ is taken as the next iterate $\mathbf{d}_1$ in our generalized steepest descent method.\footnote{From the proof of Theorem \ref{chimu}, it is clear that, $h(\mathbf{d}_1) < h(\mathbf{d}_0)$.}
\item Repeat above process iteratively. In generalized steepest descent, we can guarantee that $h(\mathbf{d}_{k+1})<h(\mathbf{d}_k)$. At every iteration, find $\nabla V(\mathbf{z}_{d_k})$ and check whether it is approximately zero or not. If in j-th iteration, we get that $\nabla V(\mathbf{z}_{d_j})=0$, then we can conclude that $V$ is not a Lyapunov function candidate. 
\end{enumerate}
\hspace{0.2in} The $h$ restricted to directions of the form $\mathbf{u}_{\beta}$ in step-3, is a function of one variable $\beta$. Therefore, local minimizer of $h$ restricted to such directions $\mathbf{u}_{\beta}$ can be easily found from its graph. In this method, we minimize the function $h$ on direction points which are obtained from the steepest descent direction of $V$ at $\mathbf{z}_{d_k}$. Therefore, we call it the generalized steepest descent method. \\
\indent We now explain what happens when $\inf\, \{h(\mathbf{d}) \mid \mathbf{d} \in S(\mathbf{0},1)\} = 0$, and $h$ does not have any local minimizer. In this case the global infimum of $h$ will not be attained; because by definition, $h$ is a strictly positive function. In above case, after certain number of iterations of the generalized steepest descent method, the graph of the function $h$ restricted to directions of the form $\mathbf{u}_{\beta}$ would look as shown in Figure \ref{h1}. As $h$ is a strictly positive function, we get a discontinuity at a point on the ray $\{\mathbf{z}_{d_k} - \beta \nabla V(\mathbf{z}_{d_k}) \mid \beta > 0\}$ corresponding to the global infimum; and hence the global infimum of $h$ is not attained. Below is an example demonstrating this case. 
\begin{figure}[ht]
  \begin{center}
    \psfrag{h}{{\scriptsize \textcolor{blue}{$h \vert_ {\{(\mathbf{z}_{d_k} - \beta \nabla V(\mathbf{z}_{d_k})-\mathbf{x}^*)/\parallel \mathbf{z}_{d_k} - \beta \nabla V(\mathbf{z}_{d_k})-\mathbf{x}^* \parallel\; \mid \; \beta \geq 0\}}$}}}
    \psfrag{z}{{\scriptsize \textcolor{blue}{$\mathbf{z}_{d_k}$}}}
    \psfrag{d}{{\scriptsize \textcolor{blue}{$\{\mathbf{z}_{d_k} - \beta \nabla V(\mathbf{z}_{d_k}) \mid \beta \geq 0\}$}}}
    \includegraphics[scale=0.45]{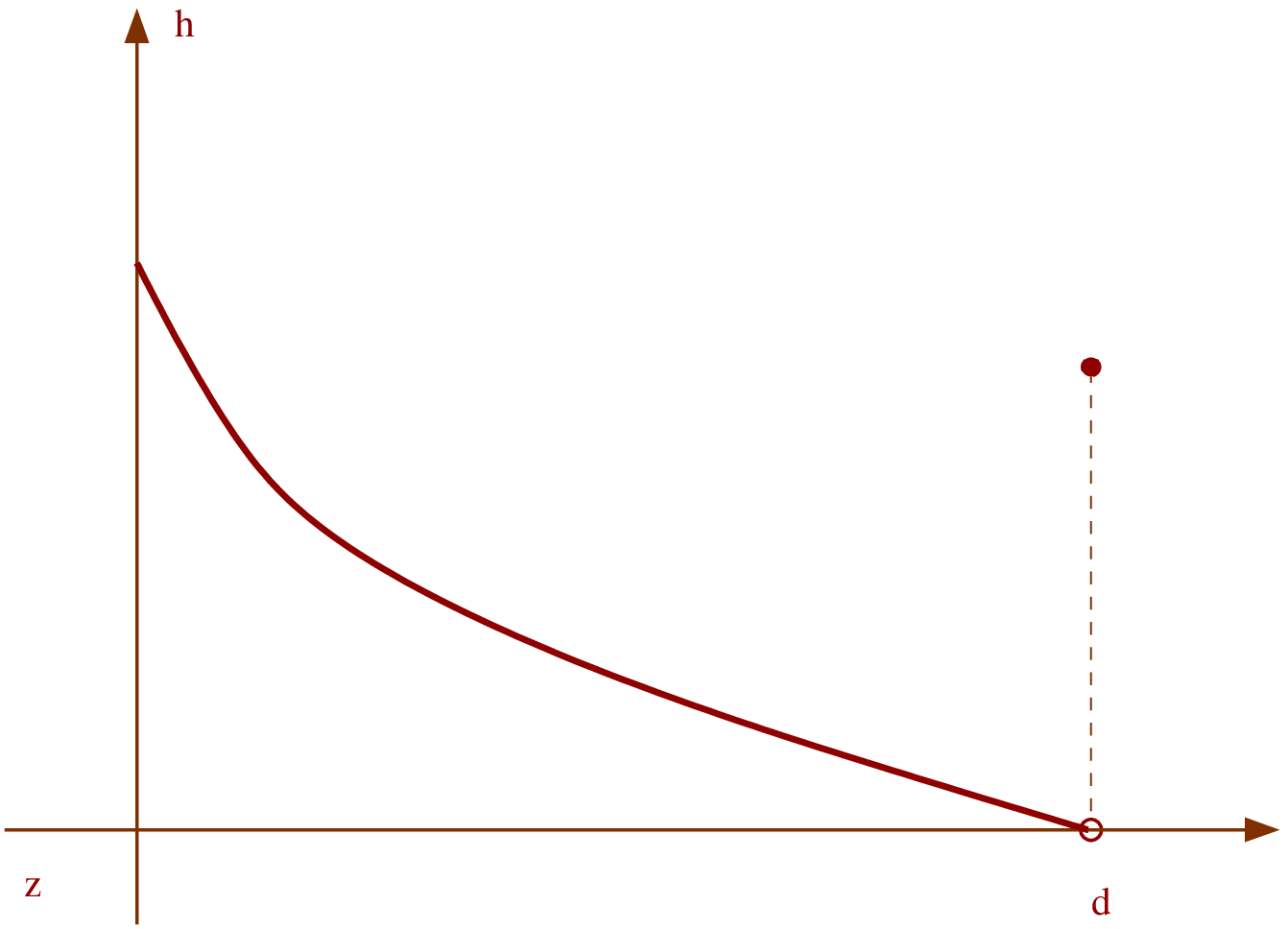}
    \caption{ }
    \label{h1}
  \end{center}
\end{figure}
\begin{example}\label{z_d zero}
 Consider a function $V:\mathbb{R}^2 \rightarrow \mathbb{R}$ defined as:
 \begin{equation}
 V(\mathbf{x}):=\frac{x_1^2}{a^2}+\frac{x_2^2}{b^2}\,. 
 \end{equation}
It can be easily checked that $V$ is positive definite and coercive. As $V$ is convex and positive definite, we can say that: for every $\mathbf{d} \in S(\mathbf{0},1)$ the one variable function $V\vert_{\{\gamma\mathbf{d}\,\mid\, \gamma \geq 0\}}$ is strictly increasing without an inflection point. In other words, for $V$ under consideration: $\inf\, \{h(\mathbf{d}) \mid \mathbf{d} \in S(\mathbf{0},1)\} = \infty $. Therefore $h$ will not have any local minimizer with finite local minimum; and hence $V$ is a Lyapunov candidate which satisfies the necessary condition obtained from Theorem \ref{chimu}.\\
 \indent Now consider a global diffeomorphism $T:\mathbb{R}^2 \rightarrow \mathbb{R}^2$ defined as follows:
 \begin{equation}
  T(\mathbf{x}) := \left[ \begin{array}{c}
                           x_1 \cos(x_1^2+x_2^2) + x_2 \sin(x_1^2+x_2^2) \\
                           x_2 \cos(x_1^2+x_2^2) - x_1 \sin(x_1^2+x_2^2)
                          \end{array} \right]\,.
\end{equation}
$T$ rotates every vector $\mathbf{x} \in \mathbb{R}^2$ in a clockwise direction by an angle $\theta(\mathbf{x})$, where $\theta: \mathbb{R}^2 \rightarrow \mathbb{R}$ is defined as $\theta(\mathbf{x}) :=x_1^2+x_2^2$. Consider a function $U:=V \circ T$,
 \begin{eqnarray}
  U(\mathbf{x}) &:=& V(T(\mathbf{x})) \nonumber \\
                &=& \frac{x_1^2 \cos^2(x_1^2+x_2^2) + x_2^2 \sin^2(x_1^2+x_2^2) + 2x_1x_2 \sin(x_1^2+x_2^2) \cos(x_1^2+x_2^2)}{a^2} + \nonumber \\ 
                & & \frac{x_2^2 \cos^2(x_1^2+x_2^2) + x_1^2 \sin^2(x_1^2+x_2^2) - 2x_1x_2 \sin(x_1^2+x_2^2) \cos(x_1^2+x_2^2)}{b^2}\,. \label{def-u}
 \end{eqnarray}
The $V$ under consideration is a continuously differentiable, positive definite, coercive function, and $T$ is a global diffeomorphism. Therefore the function $U:=V\circ T$ is also continuously differentiable, positive definite, and coercive. 
\begin{figure}[ht]
  \begin{center}
    \psfrag{t}{{\scriptsize \textcolor{blue}{$\theta$}}}
    \psfrag{s}{{\scriptsize \textcolor{blue}{$S(\mathbf{0},1)$}}}
    \psfrag{0}{{\scriptsize \textcolor{blue}{$\mathbf{0}$}}}
    \psfrag{d}{{\scriptsize \textcolor{blue}{$\mathbf{d}_{\theta}$}}}
    \includegraphics[scale=0.42]{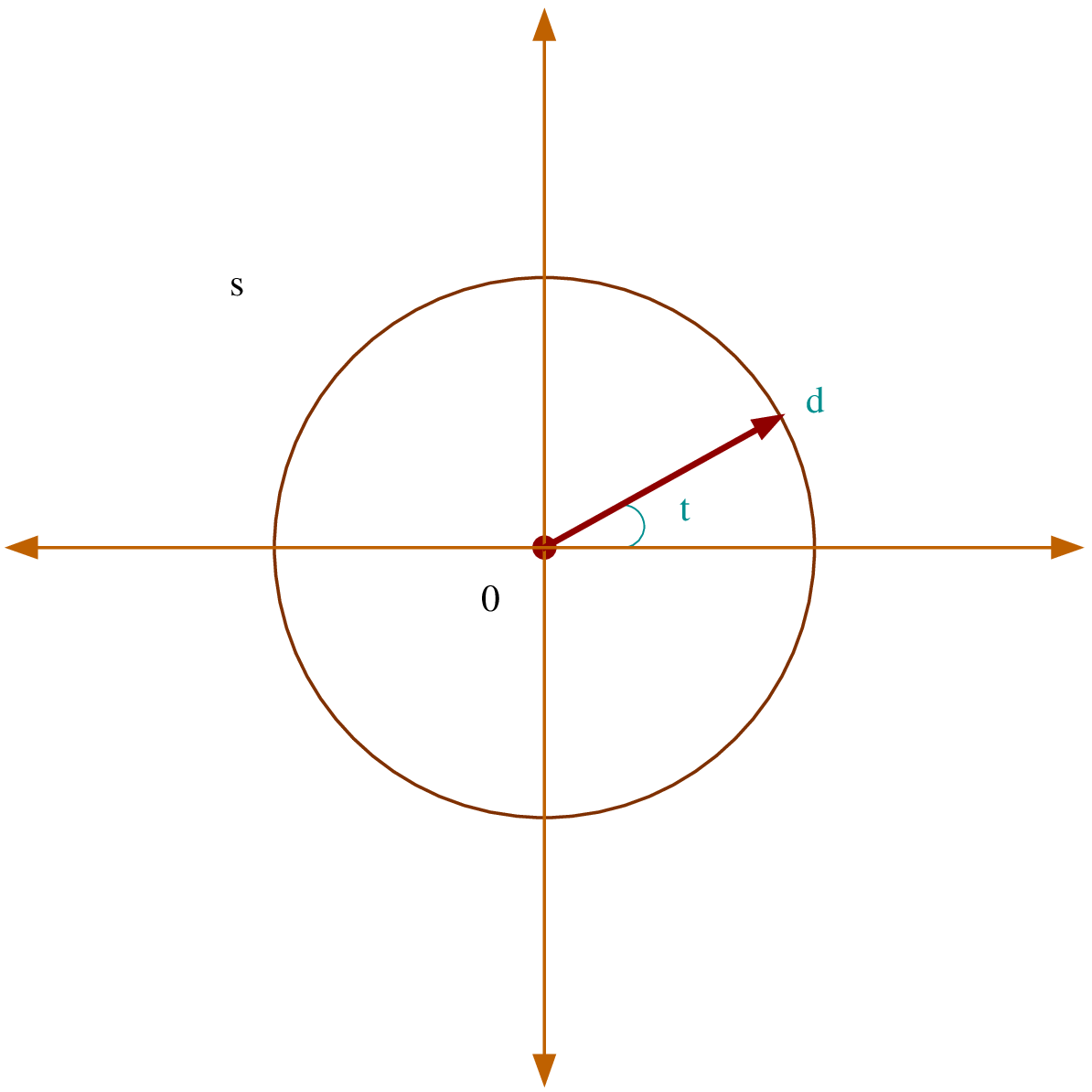}
    \caption{ }
  \label{direction}
  \end{center}
\end{figure}\\
\indent Let us parametrize directions in $\mathbb{R}^2$ by an angle $\theta$ it makes with the horizontal axis, as shown in Figure \ref{direction}. It has been checked using 'WolframAlpha' that, for the function $U:\mathbb{R}^2 \rightarrow \mathbb{R}$ given in \eqref{def-u} following holds:
\begin{eqnarray}
 \displaystyle \lim_{\theta \rightarrow 0} h(\mathbf{d}_{\theta}) &=& 0 < h(\mathbf{d}_0)\,, \\
 \displaystyle \lim_{\theta \rightarrow \pi/2} h(\mathbf{d}_{\theta}) &=& 0 < h(\mathbf{d}_{\pi/2})\,,  \\
 \displaystyle \lim_{\theta \rightarrow \pi} h(\mathbf{d}_{\theta}) &=& 0 < h(\mathbf{d}_{\pi})\,,  \\
 \displaystyle \lim_{\theta \rightarrow 3\pi/2} h(\mathbf{d}_{\theta}) &=& 0 < h(\mathbf{d}_{3\pi/2})\,.
\end{eqnarray}
\begin{figure}[ht]
\begin{center}
    \includegraphics[scale=0.42]{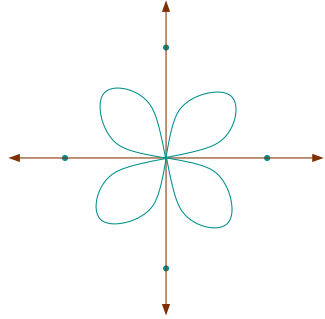}
    \caption{Approximate locus of $\mathbf{z}_{d_{\theta}}$ for $U$ given in \eqref{def-u}}
  \label{locus1}
  \end{center}
\end{figure}
Therefore for $U$ given in \eqref{def-u}, $\inf\, \{h(\mathbf{d}) \mid \mathbf{d} \in S(\mathbf{0},1)\}=0$. Corresponding to the function $U$, the approximate locus of $\mathbf{z}_{d_{\theta}}$, as $\theta$ varies in interval $[0,2\pi)$ is given in Figure \ref{locus1}. From this approximate locus it is clear that: for $U$ given in \eqref{def-u}, the function $h$ does not have any local minimizer with finite local minimum. Therefore, $U$ is a Lyapunov candidate which satisfies the necessary condition obtained from Theorem \ref{chimu}. \\
\end{example}
\indent We now explain how one could systematically search for a direction point $\mathbf{d} \in S(\mathbf{0},1)$ for which $h(\mathbf{d}) < \infty$. For this purpose, we would consider directions in $\mathbb{R}^n$ as points on the unit sphere centered at the origin with respect to $\infty$-norm rather than $2$-norm. Let us denote the unit sphere centered at the origin in $\mathbb{R}^n$ with respect to $\infty$-norm as $S_{\infty}(\mathbf{0},1)$. Imagine a grid on such unit sphere, where neighboring points are $\delta$-distance apart with respect to the $\infty$-norm; we will call such grid a $\delta$-grid (refer Figure \ref{sq} for a $\delta$-grid in $\mathbb{R}^3$). In $\mathbb{R}^n$ one could systematically move from one direction point on the $\delta$-grid to the other using $(n-1)$ loops; and this way all direction points on such $\delta$-grid can be exhausted.
\begin{figure}[ht]
  \begin{center}
   \includegraphics[scale=0.45]{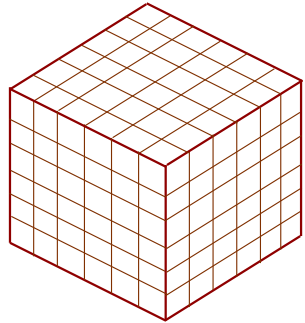}
    \caption{$\delta$-grid on the unit sphere centered at the origin w.r.t. $\infty$-norm in $\mathbb{R}^3$ }
    \label{sq}
  \end{center}
\end{figure}\\
\indent As $V$ is a continuously differentiable function, we can say the following. If for an arbitrary chosen direction $\mathbf{u}$, the one variable function $V\vert_{\{\mathbf{x}^*+\gamma\mathbf{u}\mid \gamma \geq 0\}}$ is strictly increasing without an inflection point, then there exists a neighborhood around $\mathbf{u} \in S_{\infty}(\mathbf{0},1)$ (with respect to induced topology on $S_{\infty}(\mathbf{0},1)$) such that: $V$ restricted to every direction in that neighborhood is a strictly increasing function without an inflection point. Therefore if for a sufficiently small $\delta$, $h(\mathbf{d})=\infty$ for every direction point $\mathbf{d}$ on the $\delta$-grid of $S_{\infty}(\mathbf{0},1)$; then one could say that, the following is highly probable: $h(\mathbf{d})=\infty$, $\forall\; \mathbf{d} \in S_{\infty}(\mathbf{0},1)$. In other words, in such case it is quite likely that: $V\vert_{\{\mathbf{x}^*+\gamma\mathbf{d}\mid \gamma \geq 0\}}$ is a strictly increasing without an inflection point, for every $\mathbf{d} \in S_{\infty}(\mathbf{0},1)$. If $h(\mathbf{d})=\infty$, $\forall\, \mathbf{d} \in S_{\infty}(\mathbf{0},1)$; then $V$ satisfies the necessary condition which is deduced from the Theorem \ref{chimu}. \\
\indent We explain below a simple strategy for the systematic search of a direction point $\mathbf{d} \in S_{\infty}(\mathbf{0},1)$ for which $h(\mathbf{d}) < \infty$.
 \begin{enumerate}
    \item Decide on some small value for $\delta$, and start with an arbitrary direction point $\mathbf{u}$ on $\delta$-grid of $S_{\infty}(\mathbf{0},1)$. We know how to find $h(\mathbf{u})$ for given $\mathbf{u} \in S_{\infty}(\mathbf{0},1)$. Suppose for this arbitrarily chosen $\mathbf{u} \in S_{\infty}(\mathbf{0},1)$, the function $V\vert_{\{\mathbf{x}^*+\gamma\mathbf{u}\mid \gamma \geq 0\}}$ is a strictly increasing without an inflection point, i.e. $h(\mathbf{u})=\infty$. Then, keep moving in a systematic way from one direction point on $\delta$-grid to the other till you get a direction point $\mathbf{d} \in S_{\infty}(\mathbf{0},1)$ for which $h(\mathbf{d}) < \infty$. 
    \item If during this search process you get a direction point $\mathbf{d} \in S_{\infty}(\mathbf{0},1)$ for which the one variable function $V\vert_{\{\mathbf{x}^*+\gamma\mathbf{d}\mid \gamma \geq 0\}}$ is not strictly increasing at $\mathbf{x}^*$; then $V$ does not satisfy the condition given in \eqref{pos-def}. Therefore, $V$ cannot be a Lyapunov candidate.
    \item If $h(\mathbf{d})=\infty$ for every direction point $\mathbf{d}$ on $\delta$-grid of $S_{\infty}(\mathbf{0},1)$, then it is highly probable that: $h(\mathbf{d})=\infty$, $\forall\; \mathbf{d} \in S_{\infty}(\mathbf{0},1)$. If higher accuracy is needed, then one could re-evaluate the function $h$ on a finer $\delta$-grid of $S_{\infty}(\mathbf{0},1)$. 
    \end{enumerate}
We give below some examples to show how the stated necessary condition, without involving dynamics of the system, is useful in ruling out Lyapunov candidates.
\begin{example}
Consider a scalar autonomous system $\dot{x}=-x^3$, 
and a continuously differentiable function $V$ given by $V(x)=x^6/6-13x^4/4+18x^2$. This function $V$ can be written as sum of squares as follows: \[V(x)=(\frac{x^3}{\sqrt{6}}-\frac{13\times\sqrt{6}x}{8})^2 + (18-\frac{13\times \sqrt{6}}{8})x^2\,.\]
Therefore, $V$ is positive definite and coercive; and hence it is a valid Lyapunov candidate for the autonomous system $\dot{x}=-x^3$. \\
\indent Now let us check whether $V$ satisfies the necessary condition given in section-\ref{main-result}. For vector space $\mathbb{R}$, there are only two directions: $\mathbf{d}_1=1$ and $\mathbf{d}_2=-1$. It can be checked that for $V$ under consideration, $h(\mathbf{d}_1)=h(\mathbf{d}_2)=2$. Therefore both $\mathbf{d}_1$ and $\mathbf{d}_2$ are global minimizers of $h:S(\mathbf{0},1) \rightarrow \mathbb{R}_{++} \cup \{\infty\}$.
\begin{eqnarray}
 \mathbf{z}_{d_1}=0+2\mathbf{d}_1=2 &\mbox{ and }& \mathbf{z}_{d_2}=0+2\mathbf{d}_2=-2
\end{eqnarray}
It can be checked that $V'(2)=V'(-2)=0$. Therefore, $V$ under consideration cannot be a Lyapunov function to conclude the global asymptotic stability of the origin.
\end{example} 
\indent In the above example, the vector field $f$ was a polynomial field, and hence checking the condition, $\dot{V}(\mathbf{x})=(\nabla V(\mathbf{x}))^Tf(\mathbf{x}) < 0, \; \forall \, \mathbf{x} \neq \mathbf{x}^*$ may not be too difficult. However, in general when $f$ is some complicated function, like in the example given below, checking the condition involving dynamics would not be easy. In such situation, the necessary condition which we have given in section-\ref{main-result} would be useful to rule out some Lyapunov candidates. 
\begin{example}
 Consider a second order autonomous system whose vector field is given by:
\begin{subequations}
\begin{equation}
 \dot{x}_1= \left\{ \begin{array}{cc} -\frac{x_1^2}{(1+x_1^2)}, & x_1\geqslant0 \\ \frac{x_1^2}{(1+x_1^2)}, & x_1\leqslant0 \end{array} \right.
\end{equation}
\begin{equation}
 \dot{x}_2=\left\{ \begin{array}{cc} e^{(-x_2^2)}-1, & x_2\geqslant 0 \\ 1- e^{(x_2^2)}, & x_2\leqslant 0\,. \end{array} \right.
\end{equation}
 \end{subequations}
Consider a continuously differentiable function $V$ given by $V(x_1,x_2)=x_1^6/6-13x_1^4/4+18x_1^2+x_2^2$, which can be written as sum of squares as follows: 
\begin{equation}\label{2nd-lc}
 V(x_1,x_2)=(\frac{x_1^3}{\sqrt{6}}-\frac{13\times\sqrt{6}x_1}{8})^2 + (18-\frac{13\times \sqrt{6}}{8})x_1^2+x_2^2\,.
\end{equation}
Therefore $V$ is positive definite and coercive; and hence it is a valid Lyapunov candidate for the autonomous system under consideration. 
\begin{figure}[ht]
  \begin{center}
    \includegraphics[scale=0.21]{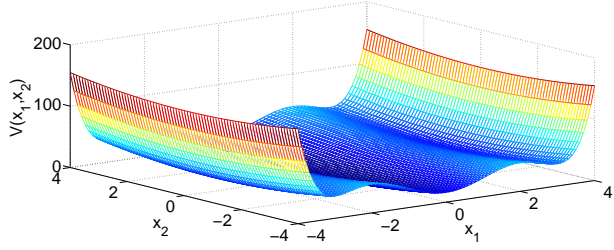}
    \caption{\small{Graph of $V(\mathbf{x})=x_1^6/6-13x_1^4/4+18x_1^2+x_2^2$ }}
  \end{center}
\end{figure} \\
\indent Now let us check whether $V$ satisfies the necessary condition given in section-\ref{main-result}. The approximate locus of $\mathbf{z}_d$ as $\mathbf{d}$ varies over sector $S_1$ and $S_3$ is shown in Figure \ref{lower-semicont}. From this approximate locus it is clear that, directions $\mathbf{e}_1$ and $-\mathbf{e}_1$ are local minimizers of  the function $h:S(\mathbf{0},1) \rightarrow \mathbb{R}_{++} \cup \{\infty\}$.
\begin{figure}[ht]
  \begin{center}
    \psfrag{s1}{{\scriptsize \textcolor{blue}{$S_1$}}}
    \psfrag{s2}{{\scriptsize \textcolor{blue}{$S_2$}}}
    \psfrag{s3}{{\scriptsize \textcolor{blue}{$S_3$}}}
    \psfrag{s4}{{\scriptsize \textcolor{blue}{$S_4$}}}
    \includegraphics[scale=0.50]{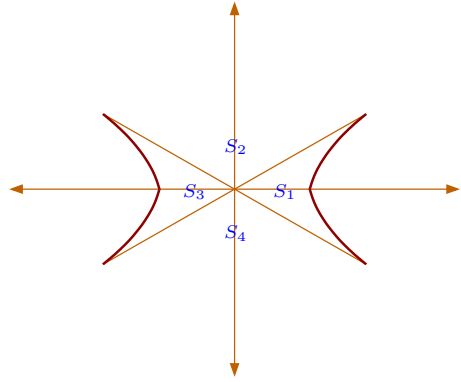}
    \caption{Approximate locus of $\mathbf{z}_d$ for $V$, as direction $\mathbf{d}$ varies over sectors $S_1$ and $S_3$.}
    \label{lower-semicont}
  \end{center}
\end{figure}\\
\indent It can be checked that, $h(\mathbf{e}_1)=h(-\mathbf{e}_1)=2$. Therefore $\mathbf{z}_{e_1}=\mathbf{0}+2\mathbf{e}_1=\left[2 \;\, 0\right]^T$, and $\mathbf{z}_{-e_1}=\mathbf{0}+2(-\mathbf{e}_1)=\left[-2\;\, 0\right]^T$.
\begin{equation}
 \nabla V(2,0) = \nabla V(-2,0) = \left[ \begin{array}{c}
                                          0 \\ 0
                                         \end{array} \right]\,.
\end{equation}
Therefore $V$  cannot be a Lyapunov function to conclude the global asymptotic stability of the origin. \vspace{0.05in} \\
\end{example}
\indent We now explain, how the generalized steepest descent method can also be used as a heuristic to check whether a continuously differentiable function $V:E \rightarrow \mathbb{R}$, where $E \subseteq \mathbb{R}^n$ with non-empty interior, is locally positive definite or not in some neighborhood $D\subseteq E$ of $\mathbf{x}^* \in E^o$. If such neighborhood $D$ exists, where $V$ satisfies \eqref{lya can}; then $V\vert_{D}$ is a Lyapunov candidate corresponding to a stable or an asymptotically stable equilibrium point $\mathbf{x}^*$. \\
\indent In this case, we have to make appropriate changes in the definition of functions $k_d$ and $h$ according to the domain $E$ of a function $V$. For the direction point $\mathbf{d} \in S(\mathbf{0},1)$, define $\alpha_d$ as follows.
\begin{equation}
 \alpha_d := \sup \{ \gamma \in \mathbb{R}_{++} \mid \mathbf{x}^*+\gamma \mathbf{d} \in E \}
\end{equation}
Denote the one variable function $V \vert_{\{\mathbf{x}^* + \gamma \mathbf{u} \,\mid\, \gamma \in [0,\alpha_d) \}}$ by $k_d$. Therefore, $k_d:[0,\alpha_d) \rightarrow \mathbb{R}$ is a map given by $k_d(\gamma) = V(\mathbf{x}^*+\gamma \mathbf{d})$. Now, define the function $h$ as follows:
\begin{equation}
 h(\mathbf{d}) := \mbox{ minimum } \gamma \in (0,\alpha_d) \mbox{ which satisfies } k_d'(\gamma) = 0\,.
\end{equation}
If such $\gamma$ does not exists, then we declare $h(\mathbf{d})=\infty$.\\
\indent We explain below, how the generalized steepest descent method can be used as a heuristic to check local positive definiteness.
\begin{itemize}
\item If $V:E \rightarrow \mathbb{R}$ is not locally positive definite for any neighborhood $D\subseteq E$ of $\mathbf{x}^* \in E^o$, then the following set of directions is non-empty.
\[
 \{ \mathbf{u} \in S_{\infty}(\mathbf{0},1) \mid V \vert_{\{\mathbf{x}^* + \gamma \mathbf{u} \,\mid\, \gamma \in [0,\alpha_u)\}} \mbox{ is not strictly increasing at point } \mathbf{x}^*\}.
\]
\item Decide on some small value of $\delta$, and evaluate the function $h$ at every point on the $\delta$-grid of $S_{\infty}(\mathbf{0},1)$. In this process, if you get a direction $\mathbf{u} \in S_{\infty}(\mathbf{0},1)$ for which $V \vert_{\{\mathbf{x}^* + \gamma \mathbf{u} \,\mid\, \gamma \in [0,\alpha_u)\}}$ is not strictly increasing at point $\mathbf{x}^*$, then $V$ is not locally positive definite for any neighborhood of $\mathbf{x}^*$. Therefore in such case, $V$ cannot be a Lyapunov candidate to conclude the stability or asymptotic stability of an equilibrium point $\mathbf{x}^*$. 
\item Suppose you don't get such direction after evaluating the function $h$ at every point on the $\delta$-grid of $S_{\infty}(\mathbf{0},1)$. Then, take the direction point $\mathbf{d}$ on the $\delta$-grid of $S_{\infty}(\mathbf{0},1)$, for which the function $h$ is minimum, as your initial iterate of the generalized steepest descent method.
\item The function $V$ under consideration is continuously differentiable, and generalized steepest descent method ensures that $h(\mathbf{d}_{k+1})\,<\,h(\mathbf{d}_k)$. Therefore if  $V$ is not locally positive definite for any neighborhood of $\mathbf{x}^*$, then the generalized steepest descent method would eventually give a direction $\mathbf{u}$ in which $V$ is not strictly increasing at point $\mathbf{x}^*$.
\end{itemize}
This heuristic is likely to give conclusive results, when $V:E \rightarrow \mathbb{R}$ is not locally positive definite for any neighborhood $D\subseteq E$ of $\mathbf{x}^* \in E^o$. Therefore, it can be used to discard functions from being locally positive definite in any neighborhood of an equilibrium point $\mathbf{x}^*$.
\section{Conclusion and Future Work}\label{con-futrwrk}
\indent We have given a necessary condition on a Lyapunov candidate to be a Lyapunov function corresponding to the globally asymptotically stable equilibrium point, which is numerically easier to check. The given necessary condition and a method to check it would be numerically useful in searching for a Lyapunov function; as it would rule out quite a few Lyapunov candidates. Theorem \ref{chimu} also has a counterpart for Lyapunov candidates corresponding to an asymptotically stable equilibrium point. This counterpart gives a necessary condition for Lyapunov functions corresponding to an asymptotically stable equilibrium point. \\
\indent Often in the process of finding a Lyapunov function to conclude the global asymptotic stability of the equilibrium point, one starts with continuously differentiable and coercive functions which can be written as sum of squares. This is because checking positive definiteness of a function is numerically very difficult. Still in order to conclude that a function in the above class is a Lyapunov function; one needs to check that, the time derivative of this function along system trajectories is negative definite. Needless to say that, checking this negative definiteness condition involving system dynamics is numerically very difficult. With a necessary condition which we 
have proposed along with the generalized steepest descent method, the set of functions on which the negative definiteness condition (involving system dynamics) needs to be checked can be made smaller. \\
\indent The generalized steepest descent method would also be useful in checking local positive definiteness of a function, which is a known necessary condition for a continuously differentiable function to be a Lyapunov function corresponding to a stable or an asymptotically stable equilibrium point. Following is a possible future work: 
 \begin{itemize}
  \item To develop a numerically more efficient method either using conjugate gradient method or some other optimization algorithm to check the necessary condition given in section-\ref{main-result}.
  \item To generalize results in this paper for time varying autonomous systems or 
  unforced systems.
 \end{itemize}
\begin{appendix}
\section{Auxiliary Results}\\
In this section, we state and prove some results related to the $h$ function.
 \begin{lemma}\label{cont-h}
  The function $h:S(\mathbf{0},1) \rightarrow \mathbb{R}_{++} \cup \{\infty\}$, defined on a Lyapunov candidate (continuously differentiable function satisfying \eqref{pos-def} and \eqref{crcv}) corresponding to the globally asymptotically stable equilibrium point $\mathbf{x}^*$, can lose its continuity only at the following types of direction points:
\begin{enumerate}
 \item At a direction point $\mathbf{d} \in S(\mathbf{0},1)$, for which there exists a sequence of direction points $(\mathbf{d}_n) \in S(\mathbf{0},1)$ converging to $\mathbf{d}$ such that, the corresponding sequence of points $\mathbf{z}_{d_n}$ converges to $\mathbf{x}^*$. 
 \item At a direction point $\mathbf{d} \in S(\mathbf{0},1)$, for which the corresponding point $\mathbf{z}_d$ is an inflection point of one variable function $V\vert_{\{\mathbf{x}^*+\gamma\mathbf{d}\mid \gamma \geq 0\}}$. 
 \item At a direction point $\mathbf{d} \in S(\mathbf{0},1)$, for which there exists a sequence of direction points $(\mathbf{d}_n) \in S(\mathbf{0},1)$ converging to $\mathbf{d}$ such that: for every $n \in \mathbb{N}$, the corresponding point $\mathbf{z}_{d_n}$ is an inflection point of one variable function $V\vert_{\{\mathbf{x}^*+\gamma\mathbf{d}_n\mid \gamma \geq 0\}}$. \vspace{0.08in}
\end{enumerate} 
 \end{lemma}
\begin{proof} We first show that, if a direction point $\mathbf{d} \in S(\mathbf{0},1)$ does not belong to any of the three mentioned cases; then the function $h$ is continuous at that direction point $\mathbf{d}$. \\
\indent As direction point $\mathbf{d}$ does not fall in any of the three mentioned cases, there exists a neighborhood $N_{\varepsilon}(\mathbf{d})$ (w.r.t. induced topology on $S(\mathbf{0},1)$) of $\mathbf{d}$ for some $\varepsilon > 0$ such that: for every $\mathbf{u} \in N_{\varepsilon}(\mathbf{d})$, the corresponding point $\mathbf{z}_u$ is not an inflection point. Moreover, the point $\mathbf{z}_d$ corresponding to direction $\mathbf{d}$ is either a non-strict local maximum (as shown in Figure \ref{graph}(a)) or a strict local maximum  (as shown in Figure \ref{graph}(b)) of the function $k_d$. As $V$ is continuously differentiable, when the direction point $\mathbf{d}$ is perturbed on $S(\mathbf{0},1)$, the $\varepsilon$-length interval shown in Figure \ref{graph} would gradually become zero for $V$ restricted to these perturbed direction points. This fact along with the fact that: for every $\mathbf{u} \in N_{\varepsilon}(\mathbf{d})$, the corresponding point $\mathbf{z}_u$ is not an inflection point, ensures that the 
function $h$ is continuous at the direction point $\mathbf{d}$.
\begin{figure}[ht]
  \begin{center}
    \psfrag{x}{{\scriptsize \textcolor{blue}{$\mathbf{x}^*$}}}
    \psfrag{a}{(a)}
    \psfrag{b}{(b)}
    \psfrag{v}{{\scriptsize \textcolor{blue}{$k_d(\gamma):=V\vert_{\{\mathbf{x}^*+\gamma\mathbf{d}\mid \gamma \geq 0\}} $}}}
    \psfrag{d}{{\scriptsize \textcolor{blue}{$\{\mathbf{x}^*+\gamma\mathbf{d}\mid \gamma \geq 0\} $}}}
    \psfrag{e}{{\scriptsize \textcolor{blue}{$\epsilon$}}}
    \psfrag{z}{{\scriptsize \textcolor{blue}{$\mathbf{z}_d$}}}
    \includegraphics[scale=0.43]{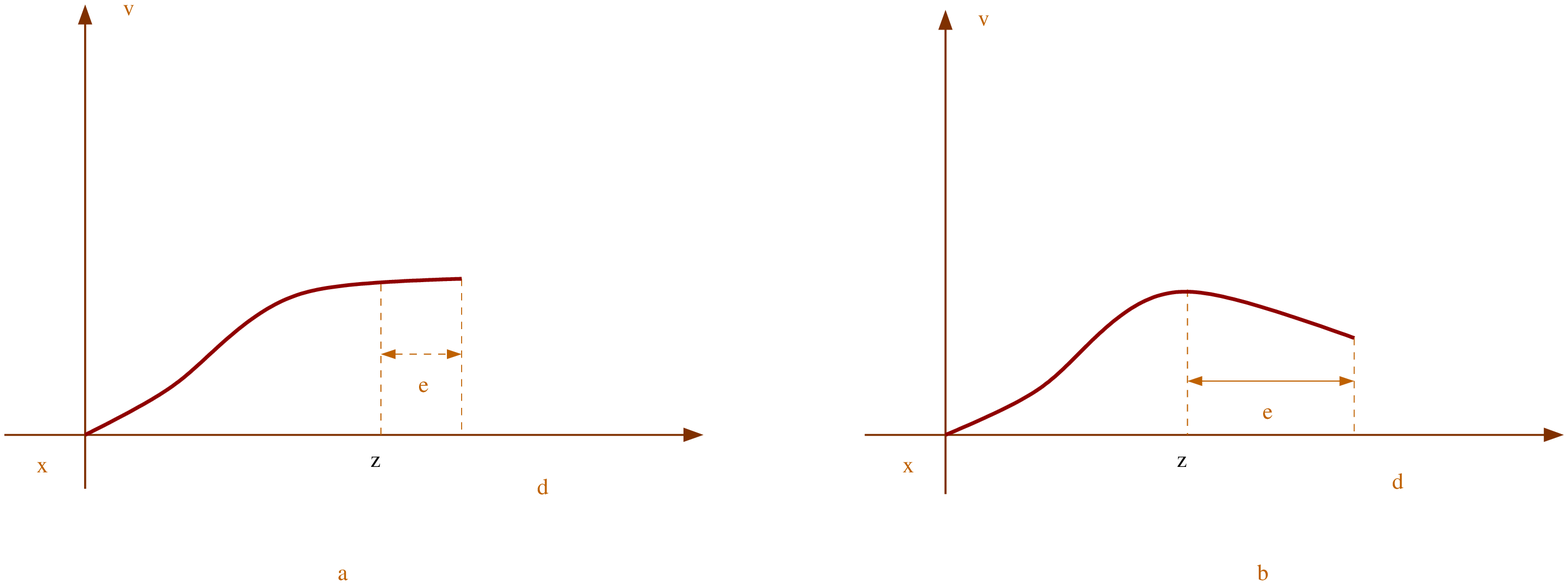}
    \caption{ }
    \label{graph}
    \end{center}
\end{figure}\\
\indent We now explain, how the function $h$ could possibly lose its continuity at direction points mentioned in the given three cases.
\begin{itemize}
\item {\large Case-1}: In this case $\displaystyle \lim_{n \rightarrow \infty} h(\mathbf{d}_n) = 0$. As range of the function $h$ is $(0,\infty]$, it would be discontinuous at $\mathbf{d} \in S(\mathbf{0},1)$ in such case.\footnote{Refer Example \ref{z_d zero}.}
\item {\large Case-2}: Consider a special instant of this case, in which there exists a neighborhood $N_{\varepsilon}(\mathbf{d})$ (w.r.t. induced topology on $S(\mathbf{0},1)$) of $\mathbf{d}$ for some $\varepsilon > 0$ such that: for every $\mathbf{u} \in N_{\varepsilon}(\mathbf{d}) \setminus \{\mathbf{d}\}$, the corresponding point $\mathbf{z}_u$ is not an inflection point. \\
Consider a sequence of direction points $\{\mathbf{d}_n\} \in N_{\varepsilon}(\mathbf{d}) \setminus \{\mathbf{d}\} $ converging to $\mathbf{d}$. For any $n \in \mathbb{N}$, the corresponding points $\mathbf{z}_{d_n}$ is not an inflection point. Therefore, $\displaystyle \lim_{n \rightarrow \infty} \mathbf{z}_{d_n}$ exists. Let $\mathbf{y}_d := \displaystyle \lim_{n \rightarrow \infty} \mathbf{z}_{d_n}$. As Lyapunov candidate $V$ is continuously differentiable, 
\begin{equation}
 \langle V(\mathbf{y}_d), \mathbf{d} \rangle = \lim_{n \rightarrow \infty} \langle V(\mathbf{z}_{d_n}), \mathbf{d}_n \rangle = 0\,.
\end{equation}
Now by definition of the $h$ function, $\parallel \mathbf{z}_d-\mathbf{x}^*\parallel\; \leq \; \parallel \mathbf{y}_d-\mathbf{x}^*\parallel$. If $\mathbf{z}_d \neq \mathbf{y}_d$, then the function $h$ would be discontinuous at direction point $\mathbf{d}$.
\item {\large Case-3}: In this case, $\mathbf{z}_{d_n}$ is an inflection point for every $n \in \mathbb{N}$. Therefore, $\displaystyle \lim_{n \rightarrow \infty} \mathbf{z}_{d_n}$ may not even exists; and even if it exists, it may not be equal to $\mathbf{z}_d$. Therefore, in this case the function $h$ could be discontinuous at direction point $\mathbf{d} \in S(\mathbf{0},1)$.
\end{itemize}
Thus, the function $h$ could lose its continuity at these types of direction points.
\end{proof} \vspace{0.1in} \\
{\large \textit{Remarks}}:
\begin{itemize}
 \item There is some intersection between case-2 type and case-3 type discontinuities in Lemma \ref{cont-h}.
 \item Note that, the case-3 type discontinuity in Lemma \ref{cont-h} has been included just for the sake of theoretical completeness. For Lyapunov candidates (corresponding to the globally asymptotically stable equilibrium point) under consideration, it is extremely unlikely that the function $h$ defined on them will have a case-3 type discontinuity at some direction point $\mathbf{d} \in S(\mathbf{0},1)$.
 \item Following lemma roughly says that: if the function $h$, defined on a Lyapunov candidate corresponding to the globally asymptotically stable equilibrium point, has a case-2 type discontinuity at some $\mathbf{d} \in S(\mathbf{0},1)$, and this discontinuity at the direction point $\mathbf{d}$ is not a case-3 type discontinuity; then the function $h$ is lower semi-continuous at $\mathbf{d} \in S(\mathbf{0},1)$.
\end{itemize}

\begin{lemma}
  Consider the function $h:S(\mathbf{0},1) \rightarrow \mathbb{R} \cup \{\infty\}$ defined on a Lyapunov candidate corresponding to the globally asymptotically stable equilibrium point $\mathbf{x}^*$. Let $\mathbf{d} \in S(\mathbf{0},1)$ be a direction point for which the corresponding point $\mathbf{z}_d$ is an inflection point of one variable function $V\vert_{\{\mathbf{x}^*+\gamma\mathbf{d}\mid \gamma \geq 0\}}$. Suppose for every sequence of direction point $\{\mathbf{d}_n\}$ converging to $\mathbf{d}$, the limit: $\displaystyle \lim_{n \rightarrow \infty} \mathbf{z}_{d_n}$ exists. Then, the function $h$ is lower semi-continuous at $\mathbf{d} \in S(\mathbf{0},1)$.
\end{lemma}
\begin{proof}
 Consider an arbitrary sequence of direction points $\mathbf{d}_n\in S(\mathbf{0},1)$ converging to $\mathbf{d}$. Let us denote $\displaystyle \lim_{n \rightarrow \infty} \mathbf{z}_{d_n}$ by $\mathbf{y}_{(d,\{d_n\})}$ (with direction $\mathbf{d}$ and sequence $\{\mathbf{d}_n\}$ as a subscript). As Lyapunov candidate $V$ is continuously differentiable, 
\begin{equation}
 \langle V(\mathbf{y}_{(d,\{d_n\})}), \mathbf{d} \rangle = \lim_{n \rightarrow \infty} \langle V(\mathbf{z}_{d_n}), \mathbf{d}_n \rangle = 0\,.
\end{equation}
Now by definition of the $h$ function, $\parallel \mathbf{z}_d-\mathbf{x}^*\parallel\; \leq \; \parallel \mathbf{y}_{(d,\{d_n\})}-\mathbf{x}^*\parallel$. Same argument can be made for every sequence of direction points converging to $\mathbf{d}$. Therefore, the function $h$ is lower semi-continuous at the direction point $\mathbf{d}$.
\end{proof}

\end{appendix}

\end{document}